\documentclass[a4paper,10pt]{amsart}
\usepackage[foot]{amsaddr}
\usepackage{amsmath} 
\usepackage{amssymb} 
\usepackage{hyperref}

\usepackage{enumitem}

\usepackage[normalem]{ulem}
\usepackage{centernot}

\newcommand{\calK}{\mathcal{K}}

\newcommand{\Rpn}{\R_{0}^{+}}

\newcommand{\C}{\mathbb{C}}

\newcommand{\N}{\mathbb{N}}
\newcommand{\R}{\mathbb{R}}
\newcommand{\Rp}{\R_{+}}

\renewcommand{\Re}{\operatorname{Re}}
\newcommand{\ee}{\operatorname{e}}
\newcommand{\vertiii}[1]{{\left\vert\kern-0.25ex\left\vert\kern-0.25ex\left\vert #1 
    \right\vert\kern-0.25ex\right\vert\kern-0.25ex\right\vert}}

\def\text{\textrm}

 \newtheorem{theorem}{Theorem}[section]

 \newtheorem{proposition}[theorem]{Proposition}
 \theoremstyle{definition}
 \newtheorem{definition}[theorem]{Definition}
 \theoremstyle{remark}
 \newtheorem{rem}[theorem]{Remark}
 \newtheorem{example}[theorem]{Example}
 \numberwithin{equation}{section}

\newcommand{\Reg}{\operatorname{Reg}}
\begin{document}

\title[]{Input-to-state stability for parabolic boundary control: Linear and Semi-linear systems}

\date{August 17, 2019}

\author[F.L.~Schwenninger]{Felix L.~Schwenninger}
\address{Department of Applied Mathematics, University of Twente, P.O.~Box 217, 
7500 AE Enschede, 
The Netherlands and \newline Department of Mathematics, Center for Optimization and Approximation, University of Hamburg, 
Bundesstr. 55, 20146 Hamburg, Germany}
\email{f.l.schwenninger@utwente.nl}

\begin{abstract}
Input-to-state stability (ISS) for systems described by partial differential equations has seen intensified research activity recently, and in particular the class of boundary control systems, for which truly infinite-dimensional effects enter the situation. This note reviews input-to-state stability for parabolic equations with respect to general $L^{p}$-input-norms in the linear case and includes extensions of recent results on semilinear equations.
\end{abstract}

 

\maketitle

\section{Introduction}

In the study of control of partial differential equations two main types of inputs can be distinguished: {\it distributed} and {\it boundary} inputs (or disturbances or controls). The latter emerge  e.g.\ by the following reason: Although a system is described by an infinite-dimensional state space, the ability to influence the system may only be possible  through an ``infinitesimal small number'' of states. As a simple motivating example consider a metal rod of length $1$ whose temperature flux at both boundary point is subject to control. Neglecting the width of the rod and normalizing parameters, the heat distribution may be governed by the following equations,
\begin{equation}\label{eq0:heateq1D}
\left\{\begin{split}
\frac{\partial x}{\partial t}(\xi,t)={}&\frac{\partial^{2}x}{\partial \xi^{2}}(\xi,t)-ax(\xi,t),&(\xi,t)\in[0,1]\times(0,\infty),\\
\frac{\partial x}{\partial\xi}(0,t)={}&\frac{\partial x}{\partial\xi}(1,t)=u(t),&t\in(0,\infty),\\
x(\xi,0)={}&x_{0}(\xi),
\end{split}\right.
\end{equation}
where $a>0$. In this setting input-to-state stability (ISS) can be understood as follows: The given `data' of the model is the initial temperature distribution $x_{0}$ and the \emph{input function} $u:\R\to U=\R$, representing the temperature flux at the boundary points. 
Since it is (physically) clear that the system is causal, that is, the solution $x$ at time $t$ does not depend on the values of the function at later values, we may ask for an estimate on the (norm of the) {\it state} $x$ at time $t$ depending on (the norms of) $u|_{[0,t]}$ and $x_{0}$. In particular, if we choose for the \emph{state space} $X=L^{2}[0,1]$, we could aim for the following type of time-space estimate for solutions to the differential equation:
\begin{equation}\label{eq0:iss}
\|x(t)\|_{X}\lesssim \ee^{-t\omega_{0}}\|x_{0}\|_{X}+\|u\|_{L^{q}(0,t)},
\end{equation}
for some fixed $q\in[1,\infty]$, $\omega_{0}>0$ and all $x_{0}\in L^{2}(0,1)$,  $u\in L^{q}(0,t)$ and $t>0$. In this (linear) case, we call the system $L^{q}$\emph{-input-to-state stable} (ISS), and in fact, the above system is $L^{q}$-ISS for the parameters $q\in(4/3,\infty]$ and $\omega_{0}\in(0,a\pi^{2}]$, see Example \ref{ex3:heatNeumann}. \newline
 In the literature, the most commonly studied ISS property is with respect to $L^{\infty}$-functions. Clearly, the notion of ``solution'' is ambiguous here, and we shall, for simplicity, confine ourselves in this introduction to classical solutions of the pde with sufficiently smooth input functions $u$. \newline Since the above example is a linear system, Estimate \eqref{eq0:iss} clearly superposes the uniform global asymptotic stability of the \emph{ internal system}, that is the dependence of $x(t)$ on $x_{0}$ in the case that $u\equiv0$, and the \emph{external stability}, that  is, the stability of $u\mapsto x(t)$ when $x_{0}=0$. This combination of stability notions lies at the heart of ISS and has proved very useful particularly for nonlinear ODE systems where this superposition principle does not hold. For a detailed overview on why this concept has become a practical tool in systems and control theory we refer to \cite{Sontag08}. Note that the linear pde case is more subtle compared to the rather trivial linear finite-dimensional situation: Imagine for instance a simple space-discretization of the above heat equation which leads to a system of the form
\begin{equation}\label{eq0:AB}
\dot{\tilde{x}}(t)=A\tilde{x}(t)+Bu(t), \quad \tilde{x}(0)=\tilde{x}_{0},
\end{equation}
where $\tilde{x}$ is vector-valued and $A, B$ are matrices of appropriate dimensions. By the variation-of-constants formula the spatially-discrete system is $L^{q}$-ISS for any $q\in[1,\infty]$ if and only if $A$ is Hurwitz. If more generally \eqref{eq0:AB} describes a system with $A$ being the generator of a strongly continuous semigroup on the (possibly infinite-dimensional) state space $X$ and $B:U\to X$ being a bounded linear operator, the corresponding assertion, that the semigroup is exponentially stable if and only if the system is $L^{q}$-ISS remains valid. Comparing this to the afore-mentioned range of $q\in (4/3,\infty]$ for which the heat equation with Neumann control at the boundary is $L^{q}$-ISS, Example \ref{ex3:heatNeumann}, reveals fundamental differences between systems of the form \eqref{eq0:AB} (with bounded input operator $B$) and the ones with boundary control. So {\it what goes wrong?\footnote{\cite{MiroWirt16a}, A.~Mironchenko and F.~Wirth.
{\it Restatements of input-to-state stability in infinite dimensions: what
  goes wrong?}
\newblock In {\em Proc. of the 22th International Symposium on Mathematical
  Theory of Networks and Systems}, pages 667--674, 2016.}}. Apparently, \eqref{eq0:heateq1D} does not fit into the framework of \eqref{eq0:AB} with \emph{bounded} $B$. Instead \eqref{eq0:heateq1D} is of the abstract form
\begin{equation}\label{eq0:BCS}
\left\{\begin{split}
\dot{x}(t)={}&\mathfrak{A}x(t),&t>0,\\
\mathfrak{B}x(t)={}&u(t),&t>0,
\\ \quad x(0)={}&x_{0}, 
\end{split}
\right.
\end{equation}
where both $\mathfrak{A}$ and $\mathfrak{B}$ are unbounded operators --- we will elaborate on the precise assumptions in Section 2. Such systems have become known as \emph{boundary control sytems}. Whereas it is formally clear that our example fits into the setting of \eqref{eq0:BCS} rather than into \eqref{eq0:AB}, it is a little less clear how ISS estimates can be assessed in this case (or to discuss existence of solutions, to begin with). However, there is a way of interpreting a boundary control system as a variant of \eqref{eq0:AB}. Although this is rather well-known in the operator theorists in systems theory, the explicit argument will be recalled in Section 2, also revealing the natural connection to weak formulations from pde's. This is also done in order to place approaches and results that were recently obtained for ISS together with more classic --- but sometimes a bit folklore --- results known in the literature.

Whereas these different view-points for linear systems are often rather subject to taste or one's background --- however, the amount of effort for obtaining ISS results may differ greatly, not only because solution concepts are intimately linked with the approach --- they (can) become crucial when considering systems governed by nonlinear pde's. In line with the introductory 1D-heat equation, one may be interested in the following semilinear system
\begin{equation}\label{eq0:nonlinearheat}
\left\{\begin{split}
\frac{\partial x}{\partial t}(\xi,t)={}&\frac{\partial^{2}x}{\partial \xi^{2}}(\xi,t)+f(x(t,\xi)),&(\xi,t)\in[0,1]\times(0,\infty)\\
\frac{\partial x}{\partial\xi}(0,t)={}&\frac{\partial x}{\partial\xi}(1,t)=u(t)&t\in(0,\infty)\\
x(\xi,0)={}&0.
\end{split}\right.
\end{equation}
where $f$ is e.g.\ of the form $f(x)=-x-x^{3}$.  In general, to account for  nonlinearities, the aimed ISS estimate has to be adapted to an inequality of the more general form 
\begin{equation}\label{eq0:iss2}
\|x(t)\|_{X}\lesssim \beta(\|x_{0}\|_{X},t)+\gamma(\|u\|_{L^{q}(0,t)}),
\end{equation}
where $\beta\in\mathcal{KL}$ and $\gamma\in\mathcal{K}_{\infty}$ being classical comparison functions from Lyapunov theory,
\begin{align*} 
\calK ={}& \{\mu:\Rpn\rightarrow \Rpn \:|\: \mu(0)=0,\, \mu\text{ continuous, }  \text{strictly increasing}\},\\
 \calK_\infty ={}& \{\gamma\in\calK \:|\:  \lim_{x\to\infty} \gamma(x)=\infty\},\\
\mathcal{L}={}&\{\theta:\Rpn\rightarrow\Rpn\:|\:\gamma \text{ continuous, } \text{strictly decreasing,} \lim_{t\to\infty}\theta(t)=0 \},\\
\mathcal{KL} = {}&\{ \beta:(\Rpn)^{2}\rightarrow\Rpn\: | \: \beta(\cdot,t)\in\calK\ \forall t \text{ and }\beta(s,\cdot)\in\mathcal{L}\ \forall s\}.
\end{align*}
 
Of course, even in the uncontrolled setting $u(t)\equiv0$, equation \eqref{eq0:nonlinearheat} is more delicate to deal with than a linear equation, both in terms of existence of solutions as well as asymptotic behaviour, but well-known \cite{Pazy83,Henry}. In particular, the ``sign'' of $f$ may be crucial for  the existence of global solutions, which is necessary for ISS. Regarding ISS, we now have typical nonlinear effects (for which ISS was originally studied for ODEs \cite{sontagISS}) blended with infinite-dimensional effects (through both the heat diffusion and the boundary control).

Recently, several steps have been made to address ISS for semilinear systems, for both distributed and boundary control, e.g.\ \cite{GuivLogeOpme19,MazePrie12,KaraKrstMiro19,ZhenZhu17,ZhenZhu18} and the references in Section \ref{sec1.1}. The employed methods are diverse --- see the section paragraph --- and it seems that a unified approach for more general systems is missing and open problems remain. In the following we try to offer yet another approach to the ISS for parabolic semilinear equations from a mere functional-analytic point of view. This, though linked in spirit  with \cite{ZhenZhu17}, generalizes to more equations of the form
\begin{equation}\label{eq0:nonlinearBCS}
\left\{\begin{split}
\dot{x}(t)={}&\mathfrak{A}x(t)+f(t,x(t)),&t>0,\\
\mathfrak{B}x(t)={}&u(t),&t>0,\\ \quad x(0)={}&x_{0}.
\end{split}
\right.
\end{equation}
 Before we summarize on the state-of-the-art in the literature, let us identify the crucial tasks in identifying ISS for a parabolic system of the form \eqref{eq0:nonlinearBCS}:
\begin{enumerate}[label=(\Roman*)]
\item Global existence (and uniqueness) of solutions to \eqref{eq0:nonlinearBCS} for $u$ in the considered function class;
\item Uniform global asymptotic stability of the undisturbed system, $u\equiv0$;
\item The  $L^{q}$-ISS estimate, \eqref{eq0:iss2}.
\end{enumerate}
The first task is classical in the study of (parabolic) pde's and is typically approached by local fix-point arguments and iteratively extending the solutions to a maximal interval and a-posteriori regularity investigations. The second step, sometimes phrased by the `geometric properties' of an evolution equation in the pde literature, is dealt with differently than in (I); with methods, such as Lyapunov functions, carefully adjusted from the finite-dimensional theory. The final step (III) is closely connected to (II) and, at least in the situations studied in the literature so far, can often be accessed by weaker arguments than the ones in (II). In particular, a local (in time) version of estimate  \eqref{eq0:iss2} does in general not suffice to guarantee global solutions. However, after having settled global existence, in Section \ref{sec4} we shall see relatively simple Lyapunov arguments which are sufficient for ISS.

This note has two goals: First and foremost we would like to survey on recent developments that fall under the concept of ISS for boundary-controlled (parabolic) evolution equations: This is done with particular care at those instances where the literature has seen results in similar spirit, but emerging from different approaches. An example of such an instance is the use of the notion of {\it admissible operators} which is classic in infinite-dimensional systems theory, but comes along with quite an operator-theoretic `flavour' compared to (direct) pde arguments. We will avoid the notion of ``admissibility'' throughout this manuscript as it is, in case of uniformly globally asymptotically stable linear systems, equivalent to ISS, \cite{JacoNabiPartSchw18}. Thus admissibility in the context of ISS is rather ``another name'' than an additional property, which, for linear equations, can be used interchangeably. 
By this, we hope to contribute to clarify on some things that may be folklore knowledge in one community, while possibly unknown in others. The author strongly believes that the fact that ISS for pde's is currently studied by view-points from different fields, such as operator theory, systems theory and control of pde's, has and has had a very positive effect on the topic. 
 Apart from this survey-character, the article 
slightly extends recent findings around ISS for semilinear equations, in particular the ones in \cite{ZhenZhu17}. This includes the goal to unify some of the approaches from the literature and or to reveal common features and difficulties. We emphasize that in contrast to the introductory example and several results in the literature, we will not restrict ourselves to spatially one-dimensional systems in the following. Thereby we hope to set the ground for coming efforts in the study of ISS estimates for pde systems, which even in the semilinear parabolic case are by far not completed. 

What this note does not cover is the link to a profound application of ISS. Instead we confine ourselves to some of the --- as we believe --- mathematical essentials and refer to the literature for important topics such as ISS feedback redesign and ISS small-gain theorems, which have had great success in finite-dimensional theory. Furthermore, ISS Lyapunov functions --- interesting from both the application and the general theory --- for which even the linear case is not completely understood yet, see \cite{JMPW18} for an interesting partial result, will not be discussed here in detail \footnote{at least not explicitely.}.

Altogether we hope to address with this article both experts in ISS for infinite-dimensional systems as well as researchers new to the field. This intention has also led to the style of the presentation which is chosen in a way that, the author hopes, is more intuitive than a plain arrangement of definitions and results. 
Like in the introduction, we will try to stick closely to some tutorial examples and develop/recap the ISS theory around them. This also means that some of the results of Section \ref{sec4} should rather be seen as a first step (or better second step after what has already been done in the literature) far from being settled conclusively. We will point out such incomplete situations and comment on difficulties. For example, one of these seems to be $L^{q}$-ISS for semilinear parabolic equations with Dirichlet boundary control, where, to the best of the authors knowledge, so far only the case $q=\infty$ has partially been resolved \cite{KaraKrstMiro19,ZhenLhacZhu18,ZhenZhu18}.

\subsection{ISS for parabolic semilinear systems --- what is known}\label{sec1.1}

As mentioned before the notion of ISS in the context of pde's has only been studied in the last ten years. However, particularly for linear systems, several results had previously been known --- at least implicitly --- by other notions arising in the control of pde's or boundary value problems. For example, for linear systems  $L^{q}$-ISS is equivalent to uniform global asymptotic stability together with $L^{q}$-admissibility --- the latter property being particularly satisfied if distributed controls are considered, see \cite{DashMiro13,JacoNabiPartSchw18,MiroWirt18a}. Therefore, classical results for $L^{2}$-admissibility, e.g.\ \cite{TucsWeis09} and $L^{q}$-admissibility, $q\in[1,\infty)$ e.g.\ \cite{HaakKuns06,Staf05,Weiss89ii}, can be applied to derive ISS for linear systems. Recall that $q\in \{1,\infty\}$ are special choices for linear systems: Whereas $q=1$ can practically only arise for distributed controls \cite{Weiss89ii}, the case $q=\infty$ is implied by any other $L^{p}$-ISS estimate with $p<\infty$. By now there are several results for general linear, not necessarily parabolic, systems for distributed and boundary control, see e.g. \cite{ArgoWitrPrie12,DashMiro10,DashMiro13,KaraKrst19book,JacoNabiPartSchw18,MiroWirt18a,MiroWirt18b} and the references therein. \newline
 In the following we concentrate on works that focus on parabolic equations. 
The assessment for particular parabolic equations, both linear and semilinear, has been studied by several authors. In \cite{DashMiro10,DashMiro13,MiroIto15,MazePrie12} (coercive) ISS-Lyapunov functions are constructed for semilinear parabolic equations with distributed control. In these references, spatially one-dimensional equations are considered with the diffusion term being the Laplacian and primarily $L^{\infty}$-ISS is shown with input functions being continuous or piecewise continuous. Boundary control (or mixed boundary and distributed control) for parabolic equations has been studied in \cite{JacoNabiPartSchw18,JacoSchwZwar19,KaraKrst16a,KaraKrst17,KaraKrstMiro19,LhacSausZhuShor19,LhacShor19,ZhenZhu17,ZhenZhu18b,ZhenZhu18}: More precisely, in \cite{KaraKrst16a,KaraKrst17} $L^{\infty}$-ISS estimates for classical solutions were proved for spatially one-dimensional linear parabolic equations where $\mathfrak{A}$ referred to a regular Sturm--Liouville differential operator and with controls acting through general Robin boundary conditions \footnote{Here, ``Robin boundary conditions'' includes Dirichlet and Neumann boundary conditions.}. The proof technique rested on a careful analysis of the solutions represented via the spectral decomposition, available in this case. In \cite[Sec.~4]{JacoNabiPartSchw18} general Riesz-spectral operators were considered and more general ISS estimates. Recently, another abstract extension of \cite{KaraKrst16a,KaraKrst17} to Riesz-spectral boundary control systems has been given in \cite{LhacShor19}, also for generalized solutions  and more generally, continuous inputs.  The assumptions used in these works, which particularly include that the differential operators have discrete spectra, are not required in \cite{JacoSchwZwar19}, where a very general class of linear parabolic equations and  inputs in $L^{\infty}$ are considered, see Theorem \ref{sec3:thm} below.  Note that all these references require finite-dimensional input spaces.\newline  Semilinear diffusion equations (with constant diffusion coefficient) in one spatial coordinate have appeared in \cite{ZhenZhu17,ZhenZhu18,ZhenZhu18b} with different scenarios of boundary control. In particular, it is shown in \cite{ZhenZhu17} that Robin boundary control which is not Dirichlet control allows for $L^{q}$-ISS estimates, $q\in[2,\infty]$, under sufficient assumptions on $f$ in order to guarantee global existence of classical solutions. We will revisit these results in the present paper and show how they generalizes to more general differential operators on higher-dimensional spatial domains. In \cite{KaraKrstMiro19} maximum principles  and their compatibility with monotonicity are used to assess $L^{\infty}$-ISS for a broad class of semilinear parabolic equations with Dirichlet boundary control and infinite-dimensional input spaces. Dirichlet control has also appeared in \cite{ZhenZhu18,ZhenZhu18b} for a viscous Burger's type equation, however with a technical assumptions on the $L^{\infty}$-norm of the input functions. We also mention a recent result in \cite[Prop.~4.1]{JMPW18} which establishes  $L^{\infty}$-ISS Lyapunov functions for parabolic boundary control problems  (and even a bit more general settings). Furthermore, we remark that also linear control systems with nonlinear (closed-loop) feedback law can be interpreted as semilinear control systems, e.g.\ \cite{Staf05}. In particular, we mention the extensive results for Lur'e systems in \cite{GuivLogeOpme19} and the prior work \cite{JayaLogeRyan08}. 
\subsection{Notation}
In the following let $\R$ and $\C$ denote the real and complex numbers  respectively and $\Rp=[0,\infty)$. The letters $X$ and $U$ will always refer to complex Banach spaces with norms $\|\cdot\|_{X}, \|\cdot\|_{U}$ where we omit the reference to the space whenever it is clear from the context. Let $I\subset[0,\infty)$ be a bounded interval. By $L^{p}(I;X)$, $p\in [1,\infty)$ we refer to the $X$-valued Lebesgue spaces of measurable, $p$-integrable functions $f:I\to X$, where the Bochner integral is used to define the vector-valued integrals. The space $W^{k,p}(I;X)\subset L^{p}(I;X)$ refers to the vector-valued Sobolev functions of order $k$. The space of essentially bounded $X$-functions is denoted by $L^{\infty}(I;X)$, the space of $X$-valued regulated functions by $\mathrm{Reg}(I;X)$ --- the closure of the step functions and the space of continuous functions by $C(I;X)$ all equipped with their natural (essential) supremum norms. Furthermore, $C^{k}(I;X)$ refers to the space of $k$-times continuously differentiable functions $f:I\to X$. By $C_{c}^{\infty}(I;X)$ we refer to the functions which are $k$-times differentiable for any $k>0$ and compactly supported in $I$. If $\mathcal{Z}(I;X)$ refers to one of the defined function spaces, then $\mathcal{Z}_{\mathrm{loc}}(\Rp;X)$ denotes the space of functions $f:\Rp\to X$ such that the restriction $f|_{I}:I\to X$ lies in $\mathcal{Z}(I;X)$ for all compact subintervals $I\subset\Rp$. We will also identify the restriction $f_{I}$ with its zero extension to $\R$ or $\Rp$. For a Banach space $Y$ let $\mathcal{L}(X,Y)$ denote the space of bounded linear operators from $X$ to $Y$. We assume the reader is familiar with  basics from strongly continuous semigroups (or ``$C_{0}$-semigroups) for which we refer to the textbooks \cite{CurtZwar95,Pazy83,Staf05,TucsWeis09}. Typically we will denote a semigroup by $T$ and its generator by $A$. The growth bound of $T$ will be denoted by $\omega_{A}$. For a Hilbert space $X$ the scalar product will be denoted by $\langle\cdot,\cdot\rangle$ and for a densely defined, closed operator $A$ on $X$, let $A^{*}$ denote the Hilbert space adjoint. \newline
The notation ``$F(x)\lesssim G(x)$'' means  that there exists a constant $C>0$, which is independent of the involved variable $x$, such that $F(x)\leq CG(x)$.

\section{A recap on ISS for linear boundary control systems}\label{sec2}
Intuitively, and in particular if one has a certain class of systems in mind, it is rather straight-forward how input-to-state stability for pde's should be defined in order to generalize the finite-dimensional theory. However, as various solution concepts such as weak, mild and strong solutions for infinite-dimensional systems exist, the following abstract definition in the language of dynamical systems seems to be natural for what we need in the following, see \cite{DashMiro13,KaraZian11} and the references therein, for similar notions in the context of ISS which have motivated the following.
\begin{definition}[Dynamical control systems]
Let $X$ and $U$ be a Banach spaces. 
Let $D\subset X\times {U}^{\Rp}$ and let $\Phi:\R_{+}\times D\to X$ be a function satisfying the following properties for any $t,h\in \R_{+}$, $(x,u),(x,u')\in D$,
\begin{enumerate}[label=(\roman*)]
	\item $\Phi(0,x,u)=x$,
	\item $(\Phi(t,x,u),u(t+\cdot))\in D$ and $\Phi(t+h,x,u)=\Phi(h,\Phi(t,x,u),u(t+\cdot))$,
	\item $(x,u|_{[0,t]})\in D$ and $u|_{[0,t]}=u'|_{[0,t]}$ implies that $\Phi(t,x,u)=\Phi(t,x,u')$.
\end{enumerate}
The mapping $\Phi$ is called \emph{semiflow} and 
\begin{itemize}
\item $X$  the \emph{state space},
\item  $U$ the \emph{input space}, 
\item $D(\Phi):=D$ the space of \emph{input data},
\item  $D_{X}(\Phi)=\{x\in X\colon \exists u \text{ such that }(x,u)\in D(\Phi)\}$  the  \emph{initial values},
\item  $D_{U}(\Phi)=\{u\in \R_{+}^{U}\colon \exists x \text{ such that }(x,u)\in D(\Phi)\}$  the  \emph{input functions}.
\end{itemize}
The triple $(X,U,\Phi)$ is called a \emph{dynamical control system}.
\end{definition}
Note that for linear systems it is often possible to ``separate'' $D(\Phi)$ in the sense that $D(\Phi)=D_{X}(\Phi)\times D_{U}(\Phi)$. However, in the case of $\Phi$ referring to the semiflow arising from the classical solutions of a boundary control system --- even in the linear case --- this is not true. 
\begin{rem}
It is debatable whether the definition of a dynamic control system (as we decided to call it here) should include any continuity assumptions on the flow. For example, as a `minimal' property, one could require that $t\mapsto \Phi(t,x,u)$ is continuous for any $(x,u)\in D(\Phi)$, as suggested e.g.\ in \cite{KaraKrstMiro19}. This condition sounds reasonable in most concrete situations involving the solution concept of the pde. However, we remark that checking this property may not be trivially satisfied even in the context of linear ISS with respect to inputs from $L^{\infty}$, see \cite{JacoNabiPartSchw18,JacoSchwZwar19}. As mentioned, several abstract settings have been introduced in the literature and the assumptions vary from one to the other. We do not  claim that our definition is more suitable than others, but it seems to be reasonable for our needs.
\end{rem}
\begin{definition}[ISS of dynamical control system]
Let $(X,U,\Phi)$ be a dynamical control system and let $q\in[1,\infty]$.
We say that the dynamical control system is \emph{$L^{q}$-input-to-state stable, $L^{q}$-ISS,} if there exist functions $\beta\in \mathcal{KL}$ and $\gamma\in \mathcal{K}$ such that
\begin{equation}
\label{eq1:iss}
	\|\Phi(t,x,u)\|_{X}\leq \beta(\|x_{0}\|_{X},t)+\gamma(\|u|_{[0,t]}\|_{{L}^{q}(0,t;U)})
\end{equation}
for all $t>0$, $(x_{0},u)\in D(\Phi)\cap (X\times  L_{\mathrm{loc}}^{q}(0,\infty;U))$.

\end{definition}
More `exotic' norms other than $L^{q}$ can be considered in the study of ISS. For instance, Orlicz spaces, a generalization of $L^{p}$-spaces, appear naturally when studying \emph{integral input-to-state stability}, a variant of ISS \cite{JacoNabiPartSchw18,NabiSchw18,JacoSchwZwar19}.
We remark that in the above definition one could more generally refrain from the completeness of the spaces $X$ and $U$. 
It is also important to keep in mind that the definition of an input-to-state stable dynamical control systems requires the global existence of solutions in time, known as `forward-completeness' of the function $\Phi$. 
Infinite-dimensional examples of dynamical control systems that are ISS can readily be given by means of linear pde systems with distributed control.

\begin{example}\label{ex1}
 Let $A$ be the generator of a strongly continuous semigroup on $X$ and $B:U\to X$ be a bounded operator. It is well-known, see e.g.\ \cite[Prop.~4.2.10]{TucsWeis09} that for any $x_{0}\in D(A)$ and $u\in W_{\mathrm{loc}}^{1,1}(\R_{+};U)$ there exists a classical solution $x:[0,\infty)\to X$ to the abstract linear equation 
 \begin{align}\label{eq1:ABbdd}
 \dot{x}={}&Ax(t)+Bu(t),\quad t>0\\
 x(0)={}&x_{0} 
 \end{align}
and by the (abstract) variation-of-constants formula, 
\begin{equation}\label{eq1:VoC}
	x(t)=T(t)x_{0}+\int_{0}^{t}T(t-s)Bu(s)\,\mathrm{ds}
\end{equation}
one sees that $(X,U,\Phi)$ with $\Phi(t,x_{0},u)=x(t)$ and $D(\Phi)=D(A)\times W^{1,1}_{\mathrm{loc}}(\R_{+};U)$, where $x$ denotes the classical solution for $x_{0}\in D(A)$, is a dynamical control system which is $L^{p}$-ISS for any $p\in[1,\infty]$ if and only if $A$ generates an exponentially stable semigroup, see e.g.\ \cite[Proposition 2.10]{JacoNabiPartSchw18}. On the other hand, if we `define' a solution only by \eqref{eq1:VoC}, which is possible for any $x_{0}\in X$ and $u\in L_{\mathrm{loc}}^{1}(\R_{+};U)$, we have that $(X,U,\Psi)$ is an $L^{p}$-ISS dynamical control system, $p\in [1,\infty]$, with semiflow $\Psi(t,x_{0},u)$ defined as the left-hand-side of \eqref{eq1:VoC} and $D(\Phi)=X\times L_{\mathrm{loc}}^{1}(\R_{+};U)$, if and only if $A$ generates an exponentially stable semigroup.  \\
For instance, this can be applied to show that that the following system is $L^{p}$-ISS for any $p\in[1,\infty]$ with $X=U=L^{2}(\Omega)$, 
\begin{align*}
	\dot{x}(\xi,t)={}&\Delta x(\xi,t)-ax(\xi,t)+u(\xi,t),&(\xi,t)\in\Omega\times (0,\infty)\\
	\frac{\partial x}{\partial \nu}(\xi,t)={}&0&(\xi,t)\in\partial \Omega\times (0,\infty)\\
	x(\xi,0)={}&x_{0}(\xi), &\xi\in\Omega
\end{align*}
where $a>0$ and $\Delta$ denotes the Laplace operator on a domain $\Omega\in \R^{n}$ with smooth boundary.
\end{example}
In Example \ref{ex1} we have seen that for a linear system with distributed control the space of initial values $D_{X}(\Phi)$ can be chosen identical to $X$  provided that $\Phi$ was extended to a more general solution concept. In fact, the ISS estimate was only assessed from the variation-of-constants formula which is a hint that this integrated version of the pde is a more natural object to study ISS estimates (of course not only ISS estimates). However, as indicated in the introduction, the System \eqref{eq0:heateq1D} does not fit into the framework of Example \ref{ex1}. Before we present a work-around to this issue, let us formalize the type of system that \eqref{eq0:heateq1D} is representing.
\begin{definition}[Linear boundary control system] Let $X$ and $U$ be Banach spaces and $\mathfrak{A}:D(\mathfrak{A})\subset X\to X$ and $\mathfrak{B}:D(\mathfrak{A})\to U$ be closed operators such that
\begin{enumerate}
\item $\mathfrak{A}|_{\ker\mathfrak{B}}$ generates a $C_{0}$-semigroup on $X$, and 
\item $\mathfrak{B}$ is right-invertible, i.e.\ there exists $B_{0}\in \mathcal{L}(U,D(\mathfrak{A}))$ with $\mathfrak{B}B_{0}=id_{U}$.
\end{enumerate}
Here and in the following, we equip $D(\mathfrak{A})$ with the graph norm $\|\cdot\|_{\mathfrak{A}}:=\|\cdot\|_{X}+\left\|\mathfrak{A}\cdot\right\|_{X}$.
Then we call both the pair $(\mathfrak{A},\mathfrak{B})$ and the formally associated set of equations
\begin{equation}\label{eq1:BCS}
\left\{\begin{split}
\dot{x}(t)={}&\mathfrak{A}x(t),&t>0,\\
\mathfrak{B}x(t)={}&u(t),&t>0,\\ \quad x(0)={}&x_{0}\in X, 
\end{split}
\right.
\end{equation}
 a (linear) \emph{boundary control system}. Given a continuous function $u:[0,\infty)\to U$ and $x_{0}\in X$, a function $x:[0,\infty)\to X$ is called a \emph{classical solution} of the boundary control system if $x\in C^{1} ([0,\infty);X)\cap C([0,\infty);D(\mathfrak{A}))$ and $x$ satisfies \eqref{eq1:BCS} pointwise.
\end{definition}
Note that the definition of a classical solution implies that $\mathfrak{B}x_{0}=u(0)$.
Let us now provide an argument for the $L^{q}$-ISS Estimate \eqref{eq0:iss} for the linear heat equation, \eqref{eq0:heateq1D}, stated in the introduction. Assume that $x:[0,\infty)\to X$ is a classical solution to the pde satisfying the boundary condition for some continuous function $u:[0,\infty)\to U$. Integration by parts then readily yields
\begin{align*}
	\frac{1}{2}\frac{\mathrm{d}}{\mathrm{d}t}\|x(t)\|_{L^{2}(0,1)}^{2}={}&\Re\langle x(t),\dot{x}(t)\rangle\\
		={}&\Re\langle x(t),\frac{\partial^{2}}{\partial \xi^{2}}x(t)-ax(t)\rangle\\
		={}&-\left\|\frac{\partial }{\partial \xi}x(t)\right\|_{L^{2}(0,1)}^{2}-a\|x(t)\|_{L^{2}(0,1)}^{2}+\Re\left(x(\xi,t)\overline{u(t)}|_{\xi=0}^{\xi=1}\right)\\
		\leq{}&-\left\|\frac{\partial }{\partial \xi}x(t)\right\|_{L^{2}(0,1)}^{2}-a\|x(t)\|_{L^{2}(0,1)}^{2}+\varepsilon\|x(t)\|_{H^{1}(0,1)}^{2}+\frac{C}{\varepsilon}|u(t)|^{2},
\end{align*}
where in the last step we used the fact that the boundary trace is a continuous linear operator from the Sobolev space $H^{1}(0,1)$ to $\C^{2}$ and where $C>0$ is some absolute constant. Therefore, by Gronwall's lemma, we conclude that for any $\omega<a$  there exists $\tilde{C}>0$ such that
\begin{equation*}
	\|x(t)\|_{L^{2}(0,1)}^{2}\leq \ee^{-\omega t}\|x_{0}\|_{L^{2}(0,1)}^{2}+\tilde{C}\int_{0}^{t}\ee^{-a(t-s)}|u(s)|^{2}\,\mathrm{d}s
\end{equation*}
and thus, by H\"older's inequality, $L^{q}$-ISS, \eqref{eq0:iss}, for any $q\in[2,\infty]$ follows. Note that an argument in this spirit has been applied in \cite{ZhenZhu17} to assess ISS even for a class of semilinear one-dimensional heat equations, provided that ``the nonlinearity behaves well'' in the above estimates --- we will be more explicit on that in Section \ref{sec4}. Let us make a few remarks on this proof: Although eventually $L^{q}$-ISS is derived for $q\in[2,\infty]$, it is essential for the argument to bound the term involving $u(t)$ such that the resulting $x(t)$ is bounded in the $H^{1}$-norm squared and consequently derive an implicit inequality in $\|x(t)\|_{L^{2}(0,1)}$. However, the result is not sharp. In fact, the considered controlled heat equation \eqref{eq0:heateq1D} is $L^{q}$-ISS for all $q\in(4/3,\infty]$. To see  this, we will rewrite the boundary control system such that  an explicit solution representation of the form \eqref{eq1:VoC} as in the distributed case can be used. Here the defining properties of a boundary control system are essential. This transformation is  a well-known technique for operator theorists in systems theory \cite{TucsWeis09,Staf05}, but appears to be a type of folklore result that is hard to find explicitly in the literature. What can be found more easily, e.g.\ in \cite{CurtZwar95}, is the so-called Fattorini trick which rewrites the boundary control system into a linear system of the form \eqref{eq1:ABbdd} with bounded operator $B$ at the price that the new input is the derivative of the initial $u$. As we are interested in $L^{q}$-estimates of the input $u$, this is undesirable. This can be overcome by an additional step: To show that this is a natural view-point, we briefly lay-out the `general Fattorini trick' in the following. Recall that the assumptions made in the definition of a boundary control system are intimately linked with semigroups and thus with \eqref{eq1:VoC}. 

Let $(\mathfrak{A},\mathfrak{B})$ be a boundary control system. Denote by $T$ the semigroup generated by $A:=\mathfrak{A}|_{\ker\mathfrak{B}}$ and let $B_{0}:U\to D(\mathfrak{A})$ be a right-inverse of $\mathfrak{B}$. A simple calculation shows that for continuously differentiable $u:[0,\infty)\to U$ and a classical solution $x$ to \eqref{eq1:BCS}, the function $z=x-B_{0}u$ solves the following differential equation
\begin{align}\label{eq1:inhomog}
	\dot{z}(t)=Az(t)+\mathfrak{A}B_{0}u(t)-B_{0}\dot{u}(t), \quad z(0)=x_{0}-B_{0}u(0),
\end{align}
in the classical sense. Note in particular that by the defining properties of $B_{0}$ we have that $x-B_{0}u\in D(A)$ if and only if $x\in D(\mathfrak{A})$ and $\mathfrak{B}x=u$. This simple reformulation, however, paves the way to derive an equation that again only depends on $u$ and not on $\dot{u}$. For that consider the representation of the solution to the inhomogeneous equation \eqref{eq1:inhomog},
\begin{equation}\label{eq1:mild1}
	z(t)=T(t)(x_{0}-B_{0}u_{0})+\int_{0}^{t}T(t-s)\mathfrak{A}B_{0}u(s)\,\mathrm{d}s -\int_{0}^{t}T(t-s)B_{0}\dot{u}(s)\,\mathrm{d}s.
\end{equation}
Note that $\mathfrak{A}B_{0}\in\mathcal{L}(U,X)$ so that the first term is well-defined even for any $u\in L_{\mathrm{loc}}^{1}(\R_{+};U)$. The second term is also well-defined, even for functions $u\in W_{\mathrm{loc}}^{1,1}(\R_{+};U)$. In order to get rid of the term $\dot{u}$ we want to (formally) integrate the second term by parts. To do so, an extension of the semigroup to a larger space $X_{-1}$ is considered. This is done to make sure that $t\mapsto T(t)x$ is differentiable for $x\in X$. For some $\lambda\in \C$ in the resolvent set of the generator $A$, $X_{-1}$ is defined to the completion of the space $X$ with respect to the norm $\|(\lambda I-A)^{-1}\cdot\|$ which is independent of $\lambda$. The semigroup uniquely extends to a strongly continuous semigroup $T_{-1}(t)$ on $X_{-1}$ with the generator $A_{-1}$ being an extension of $A$ with $D(A_{-1})=X$. For this standard procedure to define $X_{-1}$, we refer to  \cite{Pazy83,Staf05,TucsWeis09}. Thus, \eqref{eq1:mild1} and particular the integrals can be viewed in the larger space $X_{-1}$. Therefore, integration by parts
 yields
\begin{equation*}
	\int_{0}^{t}T(t-s)B_{0}\dot{u}(s)\,\mathrm{d}s=\int_{0}^{t}T_{-1}(t-s)A_{-1}B_{0}{u}(s)\,\mathrm{d}s+B_{0}u(t)-T(t)B_{0}u(0).
\end{equation*}
Inserting this in \eqref{eq1:mild1} and transforming back to $x$ gives
\begin{equation}\label{eq1:mild}
	x(t)=T(t)x_{0}+\int_{0}^{t}T_{-1}(t-s)\left[\mathfrak{A}B_{0}-A_{-1}B_{0}\right]u(s)\mathrm{d}s.
\end{equation}
We emphasize that the integral will in general only exist as a limit in $X_{-1}$ whereas its value happens to be an element of $X$ for any $t>0$ by our assumption that $x$ is a classical solution to the boundary control system. Also note that $A_{-1}B_{0}\in \mathcal{L}(U,X_{-1})$ and that $x-B_{0}\in D(A)$ is in turn equivalent to $A_{-1}x+[\mathfrak{A}B_{0}-A_{-1}B_{0}]u\in X$. All this leads to the definition of mild solutions.
\begin{definition}[mild solutions of boundary control systems]\label{def:mild}
Let $(\mathfrak{A},\mathfrak{B})$ be a boundary control system with state space $X$ and input space $U$. Let $T$ denote the semigroup generated by $A:=\mathfrak{A}|_{\ker\mathfrak{B}}$ and $B_{0}$ be a right-inverse of $\mathfrak{B}$. Let $x_{0}\in X$ and $u\in L_{\mathrm{loc}}^{1}(\R_{+};U)$. If the function $x:[0,\infty)\to X_{-1}$ defined in  \eqref{eq1:mild} takes values only in $X$, i.e., $x(t)\in X$ for all $t>0$, and $x$ is continuous from $[0,\infty)$ to $X$, then  $x$ is called a \emph{(continuous) mild solution} of \eqref{eq1:BCS}.
\end{definition}
\begin{rem}
We want to point that in the literature the notion of a mild solution may be defined in a more general way. E.g.\ in \cite{JacoNabiPartSchw18}  an arbitrary function $x:[0,\infty)\to X_{-1}$ defined by \eqref{eq1:mild} is called a mild solution, without any assumption on the range of $x$ and its continuity. Since $B\in\mathcal{L}(U,X_{-1})$, any such function will however be continuous in the weaker norm of $X_{-1}$. The assumption that a mild solution should be $X$-valued is rather natural --- not least as one models a differential equation by choosing for a norm/space initially --- and so is the continuity (in $X$). While the first one is necessary for ISS, the second (continuity) could be dropped, if we would be interested in minimal a-priori requirements for ISS estimates. However, we will see shortly that for linear systems the continuity is implicit if the system is $L^{q}$-ISS for $q<\infty$, and also for $q=\infty$, if only continuous input functions are considered, see below and \cite{JacoNabiPartSchw18,LhacShor19,TucsWeis09}. 
\end{rem}
The following properties of mild solutions corresponding to boundary control systems are well-known and can for instance be found in \cite[Chapter 11]{TucsWeis09} (in the case of Hilbert spaces). The proofs extend to the general Banach space setting in a straight-forward way, see also \cite{Staf05}. Note that in the literature there exists slightly different versions of the definition of abstract boundary control systems, e.g.\ in \cite{Grei87}. 
\begin{proposition}\label{sec2:prop1} Let $(\mathfrak{A},\mathfrak{B})$ be a boundary control system with associated operators $A$ and $B_{0}$. Let $x_{0}\in X$ and $u\in L_{\mathrm{loc}}^{1}(0,\infty;U)$. Then the following assertions holds.
\begin{enumerate}
\item Any continuous mild solution $x$ \eqref{eq1:mild} solves the equations
\begin{equation}\label{eq1:mild2}
	x(t)-x(0)=\int_{0}^{t}A_{-1}x(s)+\left[\mathfrak{A}B_{0}-A_{-1}B_{0}\right]u(s)\,\mathrm{d}s\quad t>0,
\end{equation}
where equality is understood in $X_{-1}$. Conversely, any $x\in C([0,\infty);X)$ satisfying \eqref{eq1:mild2} in $X_{-1}$ is of the form \eqref{eq1:mild} with $x_{0}=x(0)$. 

\item The operator $B=\mathfrak{A}B_{0}-A_{-1}B_{0}$ is uniquely determined by the boundary control system and does not depend on the chosen right-inverse $B_{0}$ of $\mathfrak{B}$;
\item If $x_{0}\in D(\mathfrak{A})$ and $u\in W^{2,1}(\R_{+};U)$ such that $\mathfrak{B}x_{0}=u(0)$, then there exists a unique classical solution to \eqref{eq1:BCS} given by \eqref{eq1:mild};
\item $\mathfrak{A}=A_{-1}|_{D(\mathfrak{A})}+B\mathfrak{B}$ where $B=\mathfrak{A}B_{0}-A_{-1}B_{0}\in \mathcal{L}(U;X_{-1})$;
\end{enumerate}
\end{proposition}
\begin{proof}
For the first item we refer to \cite[Theorem 3.8.2]{Staf05}. The rest can be found in \cite[Chapter 11]{TucsWeis09} upon the straight-forward adaption of proofs to general Banach spaces.
\end{proof}
In the case of Hilbert spaces (in fact, reflexive spaces suffice), we have several alternatives to characterize the operator $B$ as well as the mild solutions to a boundary control system. Note that with this one could in principle avoid the space $X_{-1}$.
\begin{proposition}\label{prop:weak}
Let the assumptions of Proposition \ref{sec2:prop1} hold and additionally assume that $X$ and $U$ are Hilbert spaces. Then the following assertions hold.
\begin{enumerate}
\item\label{1prop2} If $X$ and $U$ are Hilbert spaces, then 
$$\langle \mathfrak{A}x,\psi\rangle-\langle x,A^{*}\psi\rangle=\langle\mathfrak{B}x,B^{*}\psi\rangle\quad \forall x\in D(\mathfrak{A}), \psi\in D(A^{*}).$$

\item A continuous function $x:[0,\infty)\to X$ is a mild solution of the form \eqref{eq1:mild} if and only if it is a (weak/strong) solution in one of the following senses:
\begin{enumerate}
	\item \label{2prop2}for all $v\in D(A^{*})$ it holds that $\langle v,x(\cdot)\rangle$ is absolutely continuous and
	$$\frac{d}{dt}\langle v,x(t)\rangle=\langle x(t),A^{*}v\rangle+\langle v,\mathfrak{A}B_{0}u(t)\rangle-\langle A^{*}v,B_{0}u(t)\rangle$$
	holds for almost every $t\geq0$ and $x(0)=x_{0}$.
	\item \label{3prop2} for all $T>0$ and all $z\in C([0,T];D(A^{*}))\cap C^{1}([0,T];X)$ with $z(T)=0$ it holds that
	$$\langle z(0),x_{0}\rangle-\int_{0}^{T}\langle \dot{z}(t),x(t)\rangle \mathrm{dt}=\int_{0}^{T}\langle A^{*}z(t),x(t)\rangle+\langle z(t),\mathfrak{A}B_{0}u(t)\rangle-\langle A^{*}z(t),B_{0}u(t)\rangle\mathrm{d}t.$$
	\item \label{4prop2}$x\in W_{\mathrm{loc}}^{1,1}([0,\infty);X_{-1})$, $x(0)=x_{0}$ and $$\dot{x}(t)=A_{-1}x(t)+\left[\mathfrak{A}B_{0}-A_{-1}B_{0}\right]u(t)$$ holds in $X_{-1}$ for almost every $t\geq0$.
\end{enumerate}
\end{enumerate}
\begin{proof}
Assertion \eqref{1prop2} follows directly from \ref{sec2:prop1}, see also \cite[Remark 10.1.6]{TucsWeis09}. That the solution concept \eqref{2prop2} is equivalent to the one of a mild solution  readily follows from \eqref{eq1:mild2} in Proposition \ref{sec2:prop1} and the fundamental theorem of calculus for the Lebesgue integral, see also \cite[Remark 4.1.2]{TucsWeis09}. Also recall the duality of $X_{-1}$ and $D(A^{*})$ (see e.g.\ \cite[Prop.~2.10.2]{TucsWeis09}).\newline  Similarly, \eqref{eq1:mild2} shows the equivalence with \eqref{3prop2} by the fundamental theorem of calculus for vector-valued functions (see e.g.\ \cite{ArenBattHiebNeub11} and note that $X$ possesses the Radon--Nikodym property) and again using the duality of $X_{-1}$ and $D(A^{*})$. \newline 
To see that  \eqref{4prop2} implies \eqref{3prop2} note first that
the function $t\mapsto \langle z(t),x(t)\rangle$ is differentiable for a.e.\ $t$ and
\begin{align*}
\frac{\partial }{\partial t}\langle z(t),x(t)\rangle=&{}\langle \dot{z}(t),x(t)\rangle +\langle z(t),A_{-1}x(t)+\left[\mathfrak{A}B_{0}-A_{-1}B_{0}\right]u(t)\rangle_{D(A^{*})\times X_{-1}}\\
={}&\langle \dot{z}(t),x(t)\rangle + \langle A^{*}z(t),x(t)-B_{0}u(t)\rangle+\langle z(t),\mathfrak{A}B_{0}u(t)\rangle
\end{align*}
and thus, by integrating, $x$ satisfies the identity in \eqref{3prop2} 
for all $z\in C([0,T];D(A^{*}))\cap C^{1}([0,T];X)$ with $z(T)=0$. Conversely, assume that $x$ satisfies the condition in \eqref{3prop2} and consider $z(t)=v\tilde{z}(t)$, with $v\in D(A^{*})$, $\tilde{z}\in C^{1}([0,T];\C)$ and $\tilde{z}(T)=0$. It readily follows by the definition of the scalar-valued weak derivative and the characterization of scalar-valued Sobolev functions $W^{1,1}$ that 
$$\langle x(T),v\rangle-\langle x_{0},v\rangle=\int_{0}^{T}\langle x(t)-B_{0}u(t),A^{*}v\rangle+\langle v,\mathfrak{A}B_{0}u(t)\rangle\mathrm{d}t$$ holds. Thus, $\langle x(0),v\rangle=\langle x(\cdot),v\rangle(0)=\langle x_{0},v\rangle$ for all $v\in D(A^{*})$. Thus, by density, $x(0)=x_{0}$ and hence, \eqref{2prop2} holds. For a similar proof showing that mild solutions are weak solutions in the sense of \eqref{3prop2} see e.g. \cite[p.~631-632]{CurtZwar95} (there, however, only bounded $B$'s are considered).
  \end{proof}

\end{proposition}

\begin{rem}\label{rem1}
\begin{enumerate}
\item In \cite{LhacShor19} ISS estimates for boundary control systems are shown for continuous weak solutions in the sense of (\ref{3prop2}) of Proposition \ref{prop:weak}. There it is also shown that for smooth inputs, this definition of weak solutions coincides with solutions of the form \eqref{eq1:mild1}. In fact, as Proposition \ref{prop:weak} shows, the notions of a mild solution as introduced in Definition \ref{def:mild}, weak solutions of the form \eqref{2prop2}, \eqref{3prop2} and a ``strong solution'' \eqref{4prop2} are all equivalent provided we assume continuity. Note that the definitions of weak solutions have the advantage that they do not refer to the space $X_{-1}$.
\item  It is easy to see that the definition of classical and mild solutions can be adapted to more general boundary control systems of the form
\begin{equation}\label{eq1:BCSB}
\left\{\begin{split}
\dot{x}(t)={}&\mathfrak{A}x(t)+\tilde{B}u_{1}(t),&t>0,\\
\mathfrak{B}x(t)={}&u_{2}(t),&t>0,\\ \quad x(0)={}&x_{0}, 
\end{split}
\right.
\end{equation}
where $U_{1}$ is a Banach space, $\tilde{B}\in\mathcal{L}(U_{1},X)$ and $u_{1}:\Rp\to U_{1}$ account for some distributed control.
\item
Comparing the form of a mild solution \eqref{eq1:mild} with the usual variation-of-constants formula suggests to view a boundary control system as a special case of a system of the form 
\begin{equation}\label{eq1:AB}
	\dot{x}(t)=A_{-1}x(t)+Bu(t),\quad t>0, x(0)=x_{0}\in X,
\end{equation}
where the differential equation is understood in the larger space $X_{-1}$ for $B\in \mathcal{L}(U,X_{-1})$. Clearly, for any $x_{0}$ and $u\in L_{\mathrm{loc}}^{1}(\Rp;U)$ this equation has a unique ``mild'' solution $x:[0,\infty)\to X_{-1}$. Definition \ref{def:mild} of a mild solution for a boundary system now additionally requires that such $x$ maps indeed to $X$. Also note that this setting as the advantage that systems of the form \eqref{eq1:BCSB} are automatically encoded in that form. Conversely, if we are given a system of the form \eqref{eq1:AB} with a semigroup generator $A$ and $B\in \mathcal{L}(U,X_{-1})$, it is always possible to find operators $\mathfrak{A}:D(\mathfrak{A})\to X$, $\mathfrak{B}:D(\mathfrak{A})\to U$ and $\tilde{B}:U\to X$ so that  we have a boundary control system as in \eqref{eq1:BCSB} with $A=\mathfrak{A}|_{\ker\mathfrak{B}}$ and $B=(\mathfrak{A}-A)B_{0}$ for some (all) right-inverses $B_{0}$ of $\mathfrak{B}$. This result, in the case that $B$ is injective, can be found in \cite{Sala87}. The non-injective case can be seen upon considering the quotient space $\tilde{U}=U/\ker{B}$. In conclusion, the study of boundary control systems rather than systems \eqref{eq1:AB} is not a restriction. 
\end{enumerate}
\end{rem}
So far we have encountered --- having in mind the equivalence of Proposition \ref{prop:weak} ---  two types of solutions for boundary control systems: \emph{classical} and, more generally, \emph{mild} solutions. The use of the latter is also motivated by the fact that the objects in the ISS estimate naturally only require initial values to be in $X$ and input functions in $L^{q}$ (or the respective functions space). However, for linear systems, this choice is less `conceptual' than rather a technicality, as the following results shows. Note that the case of ISS with respect to continuous functions has already appeared in \cite{LhacShor19} (where weak solutions haven been considered instead of mild solutions). In the view of systems $(A,B)$ of the form \eqref{eq1:AB}, the following result is a simple consequence of the linearity and the density of the involved functions spaces.

\begin{proposition}[ISS w.r.t.\ different solution concepts]\label{sec2:propsol} 
Let $(\mathfrak{A},\mathfrak{B})$ be a boundary control system on a Banach space $X$ with associated operators $B_{0}$, $A$ and $B$. 
Let $q\in[1,\infty)$ and let $\Phi_{\mathrm{classic}}$ and $\Phi_{\mathrm{mild}}$ refer to the semiflow defined by the classical and mild solutions, respectively.Then the following assertion are equivalent
\begin{enumerate}
\item $(X,U,\Phi_{\mathrm{classic}})$ with $D(\Phi_{\mathrm{classic}})=D(A)\times C_{c}^{\infty}(0,\infty;U)$   is $L^{q}$-ISS; 

\item $(X,U,\Phi_{\mathrm{mild}})$ with $D(\Phi_{\mathrm{mild}})=X\times L_{\mathrm{loc}}^{q}([0,\infty);U)$ is $L^{q}$-ISS; 
\item $(X,U,\Phi_{\mathrm{classic}})$ with $D(\Phi_{\mathrm{classic}})=\{(x,u)\in D(\mathfrak{A})\times W_{\mathrm{loc}}^{2,1}(0,\infty;U)\colon \mathfrak{B}x=u(0)\}$   is $L^{q}$-ISS; 

\end{enumerate}
If $q=\infty$, then the following assertions are equivalent.
\begin{enumerate}
\item $(X,U,\Phi_{\mathrm{classic}})$ with $D(\Phi_{\mathrm{classic}})=\{(x,u)\in D(\mathfrak{A})\times W_{\mathrm{loc}}^{2,1}(0,\infty;U)\colon \mathfrak{B}x=u(0)\}$   is $L^{\infty}$-ISS; 
\item $(X,U,\Phi_{\mathrm{mild}})$ with $D(\Phi_{\mathrm{mild}})=X\times C([0,\infty);U)$ is $L^{\infty}$-ISS; 
\end{enumerate}
Note that the statements above particularly include that the considered dynamical control systems are well-defined.
\end{proposition}
\begin{proof}
Since classical solutions are mild solutions the implication (2) to (3) for $q<\infty$ and (2) to (1) when $q=\infty$ are clear.  Moreover, the implication (3) to (1) is trivial in the case $q<\infty$.  It remains to show (1) $\implies$ (2) in both regimes.\\
 Let $t>0$ be fixed and consider the operator
$$L_{t}:D(L_{t})\subset X\times L^{1}([0,t];U)\to C([0,t];X),\begin{bmatrix}x\\u\end{bmatrix}\mapsto\Phi_{\mathrm{classic}}(\cdot,x,u)|_{[0,t]}$$
with $D(L_{t})=\{(x,u|_{[0,t]})\colon (x,u)\in D(\Phi_{\mathrm{classic}})\}$. By  Proposition \ref{sec2:prop1}, $L_{t}$ is well-defined and the assumed ISS estimate together with linearity implies that $L_{t}$ is continuous with respect to the sum norm $\|x\|_{X}+\|u\|_{L^{q}(0,t;U)}$. Since classical solutions are mild solutions, Proposition \ref{sec2:prop1}, $L_t$ extends to an operator, again denoted by $L_{t}$, continuous from $D:=\{(x,u|_{[0,t]})\colon (x,u)\in D(\Phi_{\mathrm{mild}})\}$ to  $C([0,t];X_{-1})$. Consider now $q<\infty$. Thus $L_{t}$ is continuous even from  $D$ to  $C([0,t];X)$ since $D(A)\times C_{c}^{\infty}(0,\infty;U)$ lies dense in $D$ and since $X$ is continuously embedded in $X_{-1}$. For the case $q=\infty$, it may not be immediate why $D(L_{t})$ is dense in $D$. To see this, let $x\in X$ and $u\in C([0,t];U)$. Since $D(A)$ is dense in $X$, we find a sequence $(\tilde{x}_{n})_{n\ge0}$ in $D(A)=\ker\mathfrak{B}$ such that $\tilde{x}_{n}\to x-B_{0}u(0)$ for $n\to\infty$. Let $x_{n}=\tilde{x}_{n}+B_{0}u(0)$, $n\in\N$. Then $(x_{n},u)\in D(\Phi_{\mathrm{classic}})$ and $x_{n}\to x$ for $n\to\infty$. Now choose a sequence of smooth functions $u_{n}$ which satisfy $u_{n}=u(0)$ for all $n\in \N$ and approximate $u$ on $[0,t]$ in the supremum norm. It follows that $(x_{n},u_{n})\in D(\Phi_{\mathrm{classic}})$. Therefore, $L_{t}$ is continuous from $D$ to $C([0,t];X)$. From the representation \eqref{eq1:mild}, it follows that $\Phi_{\mathrm{mild}}(s,x,u)=\left(L_{t}(x,u)\right)(s)$ for any $s,t$ sucht that $s\leq t$ and $(x,u)\in D(\Phi_{\mathrm{mild}})$. 

Hence, in both cases, the continuity of the norms and the $\mathcal{KL}$, $\mathcal{K}$ functions directly gives the ISS estimates for $(X,U,\Phi_{\mathrm{mild}})$. 
 \end{proof}
\begin{rem}
\begin{enumerate}
\item The result of Proposition \ref{sec2:propsol} remains true if one replaces $L^{q}$, $q<\infty$ by the Orlicz space $E_{\Phi}$ as defined in \cite{JacoNabiPartSchw18}, since $W^{2,1}(0,t;U)$ is dense in $E_{\Phi}(0,t;U)$.
\item The proof of Proposition \ref{sec2:propsol}  can also be easily given by viewing the boundary control system as a linear system of the form \eqref{eq1:AB}. This completely reduces to the fact that an operator is bounded if and only if it is bounded (with an explicit estimate) on a dense subspace. 
\item In the view of Proposition \ref{sec2:propsol}, one could also completely avoid the space $X_{-1}$ in the above considerations  and define generalized solutions for the case that $L^{q}$ estimates are known for the classical solutions. Then $\Phi$ can be defined as abstract extension of $\Phi_{\mathrm{classic}}$ on the space $X\times L_{\mathrm{loc}}^{q}(0,\infty;U)$ or $X\times C(0,\infty;U)$ respectively, in a similar way as followed in the proof. Such solutions concepts (which coincide in this case) are known as ``generalized solutions'' in the literature, see e.g.\ \cite{SchmZwar18}, \cite[Def.~4.2]{TucsWeis14}. 
\item The proof of Proposition \ref{sec2:propsol} also shows the following: Let $q<\infty$ (for the case $q=\infty$ see below) and  $(X,U,\Phi)$ be a dynamical control system for a boundary control system with $D(\Phi)\subseteq X\times L_{\mathrm{loc}}^{q}(0,t;U)$ with the property that it extends the dynamical control system given by the classical solutions $(X,U,\Phi_{\mathrm{classic}})$ in the following sense: 
\begin{itemize}
\item $\Phi(\cdot, x,u)\in C(0;\infty;X)$ for any $(x,u)\in D(\Phi)$,
\item $(x,u)\mapsto\Phi(t,x,u)$ is linear for any $t\in\Rp$,
\item $D(\Phi_{\mathrm{classic}})\subseteq D(\Phi)$,
\item $\Phi_{\mathrm{classic}}(t,x,u)=\Phi(t,x,u)$ for all $(t,x,u)\in \Rp\times D(\Phi_{\mathrm{classic}})$.
\end{itemize}
Then, it holds that 
$$\Phi=(\Phi_{\mathrm{mild}})|_{\Rp\times D(\Phi)}$$ if  $(X,U,\Phi_{\mathrm{classic}})$ is $L^{q}$-ISS. The same assertion holds for $q=\infty$ with the modification as in Proposition \ref{sec2:propsol}. As a side effect, this provides another proof that the weak solutions considered in \cite{LhacShor19} coincide with the mild definitions defined here, at least if the dynamical control system is $L^{q}$-ISS. Note that, by Proposition \ref{prop:weak}, this holds true even without any assumption on ISS. \newline
All of this shows that ISS estimates for \emph{continuous} input functions and linear systems do ultimately not rely on the ``solution concept'', but essentially only on the classical solutions, see \cite{LhacShor19} for a similar conclusion.
\item In contrast to the previous comment in this remark, we want to point out that if one aims to study $L^{\infty}$-ISS for input functions in $L_{\mathrm{loc}}^{\infty}(0,\infty;U)$ or the regulated functions $\Reg_{\mathrm{loc}}(0,\infty;U)$, then $L^\infty$-ISS estimates for the classical solutions are not sufficient. This issue is crucial as one may want to allow for non-continuous input-functions.
\end{enumerate}
\end{rem}
Above we have seen that the regularity of the boundary trace was the key to derive the $L^{2}$-ISS estimate in the case of the toy example heat equation with Neumann boundary control. In fact, this conclusion follows from the upcoming Proposition \ref{prop:analyticB}, which will also show that a better $L^q$-ISS estimate can be obtained. Before let us recap a few essentials about parabolic equations in the view of semigroup theory. Recall that a semigroup $T$ is called \emph{analytic} if $T$ can be extended analytically to an open sector $S_{\phi}=\{x\in\C\colon|\arg z|\leq\phi\}$, $\phi\in(0,\pi)$  and \emph{bounded analytic} if $T$ is bounded on $S_{\phi}$. An important characteristic of analytic semigroups is that $\operatorname{ran}T(t)\subset D(A)$ for all $t>0$ and
\begin{equation}\label{eq1:propA}
\sup_{t>0}t\ee^{-t\omega}\|AT(t)\|<\infty
\end{equation}
for $\omega>\omega_{A}$. We will now introduce interpolation spaces $X_{\alpha}$ for analytic semigroups. Note that there are several approaches to do so and we only touch the topic very briefly here. Let us without loss of generality assume that $\omega_{A}<0$. If $A$ generates an analytic semigroup, one can define the fractional power $(-A)^{-\alpha}:X\to X$ for any $\alpha\in(0,1)$ by the contour integral
$$(-A)^{-\alpha}=\int_{\partial S_{\phi'}}z^{-\alpha}(zI+A)^{-1}\mathrm{d}z,$$
where $\partial S_{\phi'}$ is the boundary of a sufficiently large sector $S_{\phi'}$ which particularly contains the spectrum of $A$. Since $(-A)^{-\alpha}$ is a bounded injective operator on $X$, one can further define $(-A)^{\alpha}=((-A)^{-\alpha})^{-1}:\operatorname{ran}(-A)^{-\alpha}\to X$. The domain of $(-A)^{\alpha}$ equipped with the graph norm is denoted by $X_{\alpha}$. 
Analogously to the space $X_{-1}$, we can define $X_{-\alpha}$ as the completion of $X$ with respect to the norm $\|(-A)^{-\alpha}\cdot\|$. The operator $(-A)^{-\alpha}$ extends uniquely to an isometric isomorphism from $X_{-\alpha}$ to $X$ which we denote again by $(-A)^{-\alpha}$. Its inverse is the unique bounded extension of $(-A)^{\alpha}$ from $X$ to $X_{-\alpha}$.
For reflexive spaces there is an equivalent view-point of the space $X_{-\alpha}$ as the dual space of the space $X_{\alpha}^{*}$ where $X_{\alpha}^{*}$ denotes the corresponding fractional space for the dual semigroup $T^{*}$ with generator $A^{*}$ and where duality is understood in the sense of the underlying pivot space $X$, see \cite[Chapter 3]{TucsWeis09} and \cite{Weiss89ii}. One of the many basic properties of these spaces are the following (continuous) inclusions, $$X_{-1}\supset X_{-\alpha}\supset X_{-\beta}\supset X\supset X_{\beta}\supset X_{\alpha}\supset X_{1},$$ where $0<\alpha<\beta<1$.
If the growth bound of the semigroup satisfies $\omega_{A}\ge 0$, the above construction can be performed for a suitably rescaled semigroup $\ee^{-t\omega}T(t)$ and it can be shown that $X_{\alpha}$ does not depend on the chosen $\omega>\omega_{A}$. For specific examples (for example when $A$ is the Laplacian with Dirichlet boundary conditions) these abstract spaces indeed reduce to well-known fractional Sobolev spaces, which is why $X_{\alpha}$ is sometimes called an ``abstract Sobolev space''. \newline In the spirit of \eqref{eq1:propA}, the fractional powers $(-A)^{\alpha}$ of a generator of an exponentially stable analytic semigroup satisfy $\operatorname{ran}T(t)\subset D(-A^{\alpha})$ for all $t>0$ and
$$\sup_{t>0}t^{\alpha}\ee^{-t\omega}\|(-A)^{\alpha}T(t)\|<\infty,$$
for any $\omega>\omega_{A}$. Moreover, it holds that $\operatorname{ran} T_{-1}(t)\subset D(A)$ for all $t>0$. For details on interpolation spaces for analytic semigroup generators we refer e.g.\ to \cite{engelnagel,Haase06,Pazy83}.

With these preparatory comments on analytic semigroups, we can prove the following sufficient condition for ISS.
\begin{proposition}[$L^{q}$-ISS for analytic semigroups]\label{prop:analyticB}
Let $(\mathfrak{A}, \mathfrak{B})$ be a boundary control system  on a Banach space $X$ with associated operators $A$ and $B_{0}$ and $B=\mathfrak{A}B_{0}-A_{-1}B_{0}$. 
Furthermore, assume that $A$ generates an exponentially stable analytic semigroup $T$ and that one of the following properties are satisfied for some $\alpha\in (0,1]$.
\begin{enumerate}[label=(\roman*)]
\item $B_{0}\in\mathcal{L}(U,X_{\alpha})$
\item $B\in\mathcal{L}(U,X_{-1+\alpha})$ 
\item $B^{*}\in \mathcal{L}(X_{1-\alpha}^{*},U^{*})$ and $X$ is reflexive\footnote{where $X_{\beta}^{*}$ denotes the dual space of $X_{-\beta}$ with respect to the pivot space $X$}.
\end{enumerate}
Then $(\mathfrak{A},\mathfrak{B})$ is $L^{q}$-ISS for $q\in(\alpha^{-1},\infty]$. More precisely, the dynamical control system $(X,U,\Phi_{\mathrm{mild}})$ is $L^{q}$-ISS for $D(\Phi_{\mathrm{mild}})=X\times L_{\mathrm{loc}}^{q}(0,\infty;U)$, where $\Phi_{\mathrm{mild}}(t,x_{0},u)$ refers to the mild solution $x(t)$ defined in \eqref{eq1:mild}.
\end{proposition}
\begin{proof}
Either of the assumptions on $B$ imply that $(-A)^{-1+\alpha}B\in\mathcal{L}(U,X)$ and hence,
\begin{align*}\|T_{-1}(t)B\|_{\mathcal{L}(U,X)}={}&\|T_{-1}(t)(-A)^{1-\alpha}(-A)^{-1+\alpha}B\|_{\mathcal{L}(U,X)}\\
\leq{}& \|T_{-1}(t)(-A)^{1-\alpha}\|_{\mathcal{L}(X)}\|(-A)^{-1+\alpha}B\|_{\mathcal{L}(U,X)}\\
						\lesssim{}& t^{-1+\alpha}\ee^{t\omega} \|B\|_{\mathcal{L}(U,X_{-1+\alpha})}.
\end{align*}
Thus, for the H\"older conjugate $p$ of $q>(1-1+\alpha)^{-1}=\alpha^{-1}$,
\begin{align*}
\int_{0}^{t}\|T_{-1}(t-s)Bu(s)\|_{X}\,\mathrm{d}s{}&\lesssim \|B\|_{\mathcal{L}(U,X_{-1+\alpha})}  C_{q,\omega}\|u\|_{L^{q}(0,t;U)},
\end{align*}
where we used that $(-1+\alpha)p\in(-1,0)$ and $C_{q,\omega}=\|\ee^{(t-\cdot)\omega }(t-\cdot)^{(-1+\alpha)}\|_{L^{p}(0,t)}$.
Therefore the integral $\int_{0}^{t}T_{-1}(t-s)Bu(s)\,\mathrm{d}s$ converges in $X$ for any $u\in L^{q}(0,t;U)$ and the assertion follows.
\end{proof}
Recalling that the operator $B_{0}$ is not uniquely determined by the boundary control system in general, it is however easily seen that Condition (i) in the above propostion holds for \emph{all} right inverses of $\mathfrak{B}$ if and only if it holds for some $B_{0}$.
Looking at the proof, Proposition \ref{prop:analyticB} may seem rather elementary. However, it is widely applicable to settle ISS for linear parabolic boundary control problems as the assumption can often be checked by known properties of boundary trace operators.
We now come back to the discussion of the heat equation mentioned in the introduction for  general $n$-dimensional spatial domains.
\begin{example}[Heat equation with Neumann boundary control]\label{ex3:heatNeumann}
 Let $\Omega\subset \R^{n}$ be a domain with $C^{2}$-boundary $\partial \Omega$. Consider the Neumann boundary controlled heat equation with additional distributed control $d$, i.e.\
\begin{align*}
	\dot{x}(\xi,t)={}&\Delta x(\xi,t)-ax(\xi,t)+d(\xi,t),&(\xi,t)\in\Omega\times (0,\infty)\\
	\frac{\partial x}{\partial \nu}(\xi,t)={}&u(\xi,t)&(\xi,t)\in\partial \Omega\times (0,\infty)\\
	x(\xi,0)={}&x_{0}(\xi), &\xi\in\Omega.
\end{align*}
We can formulate this as a boundary control system of the form \eqref{eq1:BCSB} with 
$$X=L^{2}(\Omega),\mathfrak{A}=\Delta-aI_{X}, \mathfrak{B}=\frac{\partial }{\partial \nu},\tilde{B}=I_{X},U=L^{2}(\partial\Omega)$$
with $A=\Delta-aI_{X}$ and $D(A)=\{x\in H^{1}(\Omega):\Delta x\in L^{2}(\Omega),\frac{\partial }{\partial \nu}x=0\}$. Integrating by parts twice gives for $x\in H^{2}(\Omega)$, $\psi\in D(A)$,
$$\langle \mathfrak{A}x,\psi\rangle=\langle\mathfrak{B}x,\psi|_{\partial\Omega}\rangle_{L^{2}(\partial\Omega)}+\langle x,A\psi\rangle,$$
Since $A$ is self-adjoint, we conclude by Propositions \ref{prop:weak} that $B^{*}$ equals the boundary trace operator $\gamma_{0}$. It is known that $\gamma_{0}\in \mathcal{L}(H^{\beta}(\Omega),L^{2}(\partial\Omega))$ for any $\beta>\frac{1}{2}$, where $H^{\beta}(\Omega)$ refers to the classical fractional Sobolev space (note, however, that $\gamma_{0}$ is bounded from $H^{1}(\Omega)$ to $H^{\frac{1}{2}}(\partial\Omega)$, see \cite[13.6.1]{TucsWeis09}). In terms of the abstract Sobolev spaces $X_{\alpha}$ this means that $B^{*}\in\mathcal{L}(X_{\beta},U^{*})$ for any $\beta>\frac{1}{4}$, see e.g.\ \cite{LasiTrigI00}. Also recall that the Neumann Laplacian on  $L^{2}(\Omega)$ has spectrum in $(-\infty,0]$ which implies that $A$ generates an exponentially stable analytic semigroup as $a>0$. Thus, we can infer from Proposition \ref{prop:analyticB} that the system is $L^{q}$-ISS with respect for any $q>(1-\frac{1}{4})^{-1}=\frac{4}{3}$. Because $\tilde{B}$ is bounded from $X$ to $X$, we obtain the ISS estimates for any $q>\frac{3}{4}$ and $\tilde{q}\geq1$.
\begin{equation*}
\|x(t)\|_{L^{2}(\Omega)}\lesssim \ee^{-a t}\|x_{0}\|_{L^{2}(\Omega)}+\|u\|_{L^{q}(0,t;L^{2}(\partial\Omega))}+\|d\|_{L^{\tilde{q}}(0,t;X)}
\end{equation*}
for all $t>0$, $d\in L^{\tilde{q}}(0,t;X)$ and $u\in L^{q}(0,t;L^{2}(\partial\Omega))$. \newline
Similarly, we can consider the situation where the control does only act on a part of the boundary $\partial\Omega$, and adapt the argumentation in \cite[p.~351]{ByrnGillShubWeis02}.

\end{example}
\begin{rem}
\begin{enumerate}
\item  The author is not aware of way to sharpen the ``Lyapunov argument'' for ISS from the introduction on the Neumann controlled heat equation in order to derive the same (sharp) result $p>4/3$ as in  Example \ref{ex3:heatNeumann}. It seems that such Lyapunov arguments heavily rely on the fact that the space of input functions is $L^2$ (in time).
\item It is straight-forward to generalize Example \ref{ex3:heatNeumann} to a Neumann boundary problem for a general uniformly elliptic second-order differential operator with smooth coefficients.
\end{enumerate}
\end{rem}
Another, and in the view of the Lyapunov arguments mentioned in the introduction, more interesting example is the Dirichlet-boundary controlled heat equation.
\begin{example}[Dirichlet controlled heat equation]\label{ex:dirichlet}
Let $\Omega\subset \R^{n}$ be a domain with $C^{2}$-boundary $\partial \Omega$. The Dirichlet boundary controlled heat equation 
\begin{align*}
	\dot{x}(\xi,t)={}&\Delta x(\xi,t),&(\xi,t)\in\Omega\times (0,\infty)\\
	x(\xi,t)={}&u(\xi,t)&(\xi,t)\in\partial \Omega\times (0,\infty)\\
	x(\xi,0)={}&x_{0}(\xi), &\xi\in\Omega.
\end{align*}
can be formulated as a boundary control system with 
$$X=L^{2}(\Omega),\mathfrak{A}=\Delta, \mathfrak{B}x=x|_{\partial\Omega}, U=L^{2}(\partial\Omega)$$
with $A=\Delta$ and $D(A)=\{x\in H^{1}(\Omega):\Delta x\in L^{2}(\Omega),\gamma_{0}x=0\}$, where $\gamma_{0}$ denotes the boundary trace. Integrating by parts twice gives for $x\in D(A)$, $\psi\in C^{\infty}(\Omega)$,
$$\langle \mathfrak{A}x,\psi\rangle=\langle\mathfrak{B}x,\frac{\partial \psi}{\partial \nu}\rangle_{L^{2}(\partial\Omega)}+\langle x,A\psi\rangle,$$
Since $A$ is self-adjoint, we conclude by Proposition \ref{prop:weak} that $B^{*}$ equals the Neumann boundary trace operator $\gamma_{1}$ for which we have $\gamma_{1}\in \mathcal{L}(H^{\beta}(\Omega),L^{2}(\partial\Omega))$ for any $\beta>\frac{3}{2}$, see e.g.\ \cite[Appendix]{TucsWeis09}. In terms of the abstract Sobolev spaces $X_{\alpha}$ this means that $B^{*}\in\mathcal{L}(X_{\beta},U^{*})$ for any $\beta>\frac{3}{4}$, see e.g.\ \cite{LasiTrigI00}. Since the Dirichlet Laplacian on  $L^{2}(\Omega)$ generates an exponentially stable analytic semigroup, by Proposition \ref{prop:analyticB}  the system is $L^{q}$-ISS with respect for any $q>(1-\frac{3}{4})^{-1}=4$. Thus,
 \begin{equation*}
\|x(t)\|_{L^{2}(\Omega)}\lesssim \ee^{-\lambda_{0} t}\|x_{0}\|_{L^{2}(\Omega)}+\|u\|_{L^{q}(0,t;L^{2}(\partial\Omega))}
\end{equation*}
for all $t>0$, some $\lambda_{0}<0$ and all $u\in L^{q}(0,t;L^{2}(\partial\Omega)$.

\end{example}

As seen above, Proposition \ref{prop:analyticB} $L^{q}$-ISS for parabolic equation provided sufficient properties of the boundary operator can be shown. In concrete situations this typically reduces to knowledge of boundary traces. Let us briefly elaborate on what can be said in situations where this information is not accessible. Furthermore, one may also ask the question whether at all boundary systems exist which are not $L^{q}$-ISS for some finite $q$. Let us first answer this positively with a, admittedly pathologic, example.
\begin{example}\label{ex32}
Let $X=\ell^{2}(\N)$ be the space of complex-valued, square-summable sequences and let $(e_{n})_{n\ge1}$ denote the canonical orthonormal basis. Define $\mathfrak{A}:D(\mathfrak{A})\to X$ and $\mathfrak{B}:D(\mathfrak{A})\to\C$ by
\begin{align*}
D(\mathfrak{A})={}&\left\{ x\in\ell^{2}(\N)\colon\exists c_{x}\in\C \text{ such that }\sum_{n=1}^{\infty}\left|-2^{n}\langle x,e_{n}\rangle+\frac{c_{x}2^{n}}{n}\right|^{2}<\infty\right\}\\
\mathfrak{A}x={}&\sum_{n=1}^{\infty}\left(-2^{n}\langle x,e_{n}\rangle e_{n}+\frac{c_{x}2^{n}}{n}\right)\\
\mathfrak{B}x={}&c_{x}
\end{align*}
To see that $\mathfrak{A}$ and $\mathfrak{B}$ are well-defined, suppose that $x\in X$ and $c_{x},\tilde{c}_{x}\in \mathbb{C}$. Then it holds that 
$$\sum_{n=1}^{\infty}\left|-2^{n}\langle x,e_{n}\rangle+\frac{c_{x}2^{n}}{n}\right|^{2}<\infty, \quad \sum_{n=1}^{\infty}\left|-2^{n}\langle x,e_{n}\rangle+\frac{\tilde{c}_{x}2^{n}}{n}\right|^{2}<\infty.$$
By triangle inequality (in $\ell^{2}(\mathbb{N})$), it follows that $|c_{x}-\tilde{c}_{x}|^{2}\sum_{n=1}^{\infty}\frac{2^{2n}}{n^2}<\infty$, and hence $c_{x}=\tilde{c}_{x}$. Similarly, it follows that both operators are linear. Since $(\frac{1}{n})_{n\in\N}\in\ell^{2}(\N)$, it is clear that $\mathfrak{B}$ possesses  a right-inverse, e.g.\ given by 
$B_{0}c=c\sum_{n=1}^{\infty}\frac{1}{n}e_{n}$, $c\in \C$. The operator $A=\mathfrak{A}|_{\ker{\mathfrak{B}}}$ is given through $Ae_{n}=-2^{n}e_{n}$, $n\in\N$ on its maximal domain. This operator generates an exponentially stable, analytic semigroup $T$ determined by $T(t)e_{n}=\ee^{-2^{n}t}$, $n\in\N$. Thus, $(\mathfrak{A},\mathfrak{B})$ constitute a boundary control system. The operator $B=\mathfrak{A}B_{0}-A_{-1}B_{0}$ thus becomes $Bc=-c\sum_{n=1}^{\infty}\frac{2^{n}}{n}e_{n}$, which has to be interpreted as an operator from $\C$ to $X_{-1}=\{\sum_{n=1}^{\infty} x_{n}e_{n}\colon (\frac{x_{n}}{2^{n}})_{n\in\N}\in\ell^{2}(\N)\}$. In \cite[Example 5.2]{JacoNabiPartSchw18}, which in turn was based on a result from \cite{JacoPartPott14}, it was shown that the system $\Sigma(A,B)$ of the form \eqref{eq1:AB} is not $L^{q}$-ISS for any $q<\infty$. In particular, this implies that for any $q<\infty$ there exists a time $t_{0}$ and a sequence of continuously differentiable functions $u_{m}:[0,\infty)\to\C$ such that 
\begin{itemize}
\item $\sup_{m\in\N} \|u_{m}\|_{L^{q}(0,t_{0})}<\infty$ and
\item the classical solution $x_{m}:[0,\infty)\to X$ to the boundary control system $(\mathfrak{A},\mathfrak{B})$ with initial value $x_{0}=0$ and input function $u_{m}$ satisfy
$$\lim_{m\to\infty}\|x_{m}(t_{0})\|_{X}\to\infty.$$
\end{itemize}
However, the boundary control system is $L^{\infty}$-ISS, by the upcoming Theorem \ref{sec3:thm}.
\end{example}

\begin{theorem}\label{sec3:thm}
Let $(\mathfrak{A},\mathfrak{B})$ be a boundary control system on a Hilbert space with associated operator $A$. If the following assumptions are satisfied,
\begin{itemize}
	\item $A$ generates an exponentially stable, analytic semigroup, and
	\item there exists an equivalent scalar product $\langle\cdot,\cdot\rangle_{new}$ on $X$ such that $A$ is dissipative, i.e. $\Re\langle Ax,x\rangle_{new}\leq0$,
	\item the range of $\mathfrak{B}$ is finite-dimensional,
\end{itemize} 
then $(\mathfrak{A},\mathfrak{B})$ is $L^{\infty}$-ISS and the (mild) solutions are continuous for all $(x_{0},u)\in X\times L_{\mathrm{loc}}^{\infty}(\Rp;U)$. \newline 
Moreover, there exist positive constants $C_{1}, C_{2},\omega,\epsilon$ and a strictly increasing, smooth, convex function $\Phi:[0,\infty)\to[0,\infty)$ with $\Phi(0)=0,\lim_{x\to\infty}\frac{\Phi(x)}{x}=\infty$ such that
\begin{equation}\label{eq:thm1}
\|x(t)\|_{X}\leq C_{1}\ee^{-\omega t}\|x_{0}\|_{X}+C_{2}\ee^{-\epsilon t}\inf\left\{k\geq0\colon\int_{0}^{t}\Phi\left(\tfrac{\ee^{s\epsilon}\|u(s)\|_{U}}{k}\right)\mathrm{d}s\leq1\right\}
\end{equation}
for any  mild solution $x$, $t>0$, $u\in L_{\mathrm{loc}}^{\infty}(0,\infty)$ and $x_{0}\in X$.
\end{theorem}
\begin{proof}
This is a direct consequence of the results in \cite{JacoSchwZwar19} where  similar results were stated for systems of the form \eqref{eq1:AB}. It remains to observe the following. Because of the assumed dissipativity, the semigroup $T$ is similar to a contraction semigroup. Since $\mathfrak{B}$ has a right-inverse, it follows that $\dim U=\dim \operatorname{ran}\mathfrak{B}<\infty$. Hence, $B=(\mathfrak{A}-A)B_{0}$ is an operator from a finite-dimensional space to $X_{-1}$. 
In order to derive Estimate \eqref{eq:thm1}, we use a rescaling argument: Let $\epsilon>0$ such that $\tilde{T}=\ee^{\epsilon\cdot }T$ is exponentially stable and consider the boundary control system  $(\mathfrak{A}+\epsilon I, \mathfrak{B})$. Note that the spaces $X_{-1}$ and the corresponding one for $A+\epsilon I$, the generator of $\tilde{T}$,  coincide and also $B=(\mathfrak{A}-A_{-1})B_{0}=(\mathfrak{A}+\epsilon I- A_{-1}-\epsilon I)B_{0}$. Corollary 21 and  Theorem 19 from \cite{JacoSchwZwar19} show that there exist positive constants $\tilde{C}_{1},\tilde{C}_{2}$ and ${\omega}$ and a function $\Phi$ with the properties described in the statement of the theorem such that
\begin{equation}
\label{eq:thm2}\|\tilde{x}(t)\|_{X}\leq \tilde{C}_{1}\ee^{-\omega t}\|x_{0}\|_{X}+\tilde{C}_{2}\inf\left\{k\geq0\colon\int_{0}^{t}\Phi\left(\frac{\|\tilde{u}(s)\|_{U}}{k}\right)\mathrm{d}s\leq1\right\}
\end{equation}
for any mild solution $\tilde{x}$ of $(\mathfrak{A}+\epsilon I,\mathfrak{B})$ and $t>0$, $\tilde{u}\in L_{\mathrm{loc}}^{\infty}(0,\infty)$ and $x_{0}\in X$.  On the other hand it follows from the representation \eqref{eq1:mild}, that any mild solution $x$ to $(\mathfrak{A},\mathfrak{B})$ with input function $u$, the function $\tilde{x}(t)=\ee^{\epsilon t}x(t)$ defines a mild solution of the boundary control problem $(\mathfrak{A}+\epsilon I, \mathfrak{B})$ with input function $\tilde{u}=\ee^{\cdot\epsilon}u$. Combining this with \eqref{eq:thm2} shows \eqref{eq:thm1}. \\
To see that  \eqref{eq:thm1} implies that $(\mathfrak{A},\mathfrak{B})$ is $L^{\infty}$-ISS, we show that there exists a constant $C_{3}$ such that for all $t>0$,
$$\int_{0}^{t}\Phi\left(\frac{\ee^{s\epsilon}\|u(s)\|_{U}}{C_{3}\ee^{\epsilon t}\|u\|_{L^{\infty}(0,t;U)}}\right)\mathrm{d}s\leq 1.$$
Since $\Phi$ is strictly increasing it thus suffices to show that $$\sup_{t>0}\int_{0}^{t}\Phi\left(\ee^{-s\epsilon}C_{3}^{-1}\right)\mathrm{d}s\leq 1,$$
which follows easily by the property that $\lim_{x\to0}\frac{\Phi(x)}{x}=0$.
 
\end{proof}
\begin{rem}[on Theorem \ref{sec3:thm}]
\begin{itemize}
\item Let us point out that \eqref{eq:thm1} is indeed stronger than the corresponding estimate with $\epsilon=0$: By monotonicity of $\Phi$, 
\begin{align*}\ee^{-\epsilon t}\inf\left\{k\geq0\colon\int_{0}^{t}\Phi\left(\tfrac{\ee^{s\epsilon}\|u(s)\|_{U}}{k}\right)\mathrm{d}s\leq1\right\}\leq
\inf\left\{k\geq0\colon\int_{0}^{t}\Phi\left(\tfrac{\|u(s)\|_{U}}{k}\right)\mathrm{d}s\leq1\right\}.
\end{align*}
 Furthermore, in case that $\Phi$ can be chosen as $\Phi(x)=x^{q}$, $x\in[0,\infty)$ for $q\in(1,\infty)$, the estimate reduces to an $L^{q}$-ISS estimate. 
\item The BCS in Example \ref{ex32} satisfies the assumptions of Theorem \ref{sec3:thm} as can be checked by the explicit expression for the semigroup. However, the function $\Phi$ cannot be taken of the form $\Phi(x)=x^{q}$ for any $q<\infty$, \cite[Example 5.2]{JacoNabiPartSchw18}. 
\item The assumption that there exists an equivalent scalar product such that $A$ is dissipative is rather weak from a practical point of view: Most known practically-relevant examples of differential operators satisfy this condition, \cite{KunstWeis04} which can be rephrased as the property that the semigroup is similar to a contraction semigroup. However, it is not difficult to construct counterexamples assuring that not every analytic semigroup on a Hilbert space is similar to a contractive one. This can be done by diagonal operators with respect to (Schauder) basis which is not a Riesz basis, \cite{BailClem91,Haase06}.
\end{itemize}
\end{rem}

\begin{example}
Let $(\mathfrak{A},\mathfrak{B})$ be a boundary control system with $\dim U<\infty$ and $A$ being a Riesz-spectral operator, i.e.\ $A=S^{-1}\Lambda S$ for a bijective operator $S\in \mathcal{L}(X)$ and a densely defined closed operator $\Lambda: D(\Lambda)\subset X\to X$ with discrete spectrum $\sigma(\Lambda)$ contained in a left-half-plane of the complex plane and such that the eigenvectors establish an orthonormal basis of $X$, where we also assume that the eigenvalues are pairwise distinct. By Parseval's identity, it follows that $A$ is dissipative with respect to the scalar product $\langle S\cdot,S\cdot\rangle$. If moreover, it is assumed that $\sigma(A)=\sigma(\Lambda)$ is contained in a sector $S_{\theta}:=\{z\in\C\colon \arg(z)\leq\theta\}$ with $\theta<\frac{\pi}{2}$, then $A$ generates an analytic semigroup, which is exponentially stable if and only if $\sup\{\Re\lambda\colon\lambda\in\sigma(A)\}<0$. For details on Riesz-spectral operator we refer for instance to \cite{CurtZwar95}. Therefore, the conditions of Theorem \ref{sec3:thm} are satisfied and $(\mathfrak{A},\mathfrak{B})$ is $L^{\infty}$-ISS for input data $(x_{0},u)\in X\times L_{\mathrm{loc}}^{\infty}(\Rp;U)$. See also \cite{LhacShor19} and \cite{JacoNabiPartSchw17,JacoNabiPartSchw18} for different proofs of this fact. In particular in the latter, more generally $q$-Riesz-spectral operators are considered.
\end{example}

\section{A primer on semilinear boundary control systems}\label{sec4}
In the following we extend the linear systems considered in Section \ref{sec2} to semilinear ones. As motivating example serves \eqref{eq0:nonlinearheat}. The abstract theory of semilinear pde's (without controls/disturbances) with our without using semigroups is comparably old and can be found e.g.\ in the textbooks \cite{Henry,Pazy83}. There is a particularly rich theory for parabolic equations as smoothening effect of the linear part through the analytic semigroups allows for rather general nonlinearities. In the following we are interested in ISS estimates similar to the ones we derived for linear systems: This includes the property that the undisturbed system is uniformly asymptotically stable which requires already restrictive conditions on the nonlinearity, particularly, if we  aim for abstract results covering whole classes of examples. The simplest condition guaranteeing this global stability is a global Lipschitz condition with sufficiently small Lipschitz constant, as we shall see in Theorem \ref{sec4:thm1}. There it is shown that the usual proof technique to assess uniform global asymptotic stability for uncontrolled systems also goes through for boundary control systems provided the results we discussed in Section II. The final result of this section is Theorem \ref{thm:main}, which provides a generalization of the findings in \cite{ZhenZhu17}.

\begin{definition}[Semilinear boundary control system]
Let $(\mathfrak{A},\mathfrak{B})$ be a linear boundary control system with state space $X$ and input space $U$. Denote by $A$ the associated semigroup generator and by  $B_0$ a right-inverse of $\mathfrak{B}$. Further let
\begin{itemize}
\item $\alpha\in[0,1)$ if $A$ generates an analytic semigroup, or 
\item $\alpha=0$ else (in which case we set $X_{0}=X$).
\end{itemize}
 Let $f:\Rp\times X_{\alpha}\to X$ be a function continuous in the first variable and locally Lipschitz in the second variable with respect to the norm $X_{\alpha}$. Then the triple $(\mathfrak{A},\mathfrak{B},f)$ formally representing the equations
\begin{equation}\label{eq2:nonlinearBCS}
\left\{\begin{split}
\dot{x}(t)={}&\mathfrak{A}x(t)+f(t,x(t)),\\
\mathfrak{B}x(t)={}&u(t),&\\ \quad x(0)={}&x_{0}, 
\end{split}
\right.
\end{equation}
$t>0$, is called a \emph{semilinear boundary control system}.\newline
Let $x_{0}\in D(\mathfrak{A})$, $T>0$ and $u\in C([0,T];U)$.  A function $$x\in C([0,T];D(\mathfrak{A}))\cap C^{1}([0,T];X)$$ is called a \emph{classical solution} to the nonlinear BCS  \eqref{eq2:nonlinearBCS} on $[0,T]$ if $x(t)\in X_{\alpha}$ for all $t>0$ and the equations \eqref{eq2:nonlinearBCS} are satisfied pointwise for $t\in (0,T]$. A function $x:[0,\infty)\to X$ is called \emph{(global) classical solution} to the BCS, if $x|_{[0,T]}$ is a classical solution on $[0,T]$ for every $T>0$. If $x\in C([0,T];D(\mathfrak{A}))\cap C^{1}((0,T];X)$ and $x(t)\in X_{\alpha}$ for all $t>0$ and the equations \eqref{eq2:nonlinearBCS} are satisfied pointwise for $t\in (0,T]$, then we say that $x$ is a classical solution on $(0,T]$.

\end{definition}
Similar as in the previous section, we can define mild solutions. 
\begin{definition}[Mild solutions of semilinear boundary control systems]\label{def:mildnonlinear}
Let $(\mathfrak{A},\mathfrak{B},f)$ be a semilinear boundary control system with associated $A,B_{0},\alpha\in[0,1)$. Let $x_{0}\in X$,  $T>0$ and $u\in L_{\mathrm{loc}}^{1}([0,T];U)$. A continuous function $x:[0,T]\to X$ is called \emph{mild solution} to the BCS \eqref{eq2:nonlinearBCS} on $[0,T]$ if $x(t)\in X_{\alpha}$ for all $t>0$ and $x$ solves 
\begin{align}
x(t)={}&T(t)x_{0}+\int_{0}^{t}T(t-s)\left[f(s,x(s))+Bu(s)\right]\mathrm{d}s,\label{eq:mildsolution}
\end{align}
for all $t\in[0,T]$ and where $B=\mathfrak{A}B_{0}-A_{-1}B_{0}$. A function $x:[0,\infty)\to X$ is called a \emph{global mild solution} if $x|_{[0,T]}$ is a mild solution on $[0,T]$ for all $T>0$.
\end{definition}
It is not hard to see that the definition \ref{def:mildnonlinear} coincides with the one for linear BCS in case that $f(t,x)=Cx$ for any bounded operator $C:X\to X$, or, more generally, when $C$ is unbounded and $A+C$ generate a strongly continuous semigroup. Moreover, any (global) classical solution is a  (global) mild solution. 

The following result is not very surprising as it shows that a semilinear system is ISS if the linear subsystem is ISS and the nonlinearity is globally Lipschitz. 
\begin{theorem}\label{sec4:thm1}
Let $(\mathfrak{A},\mathfrak{B})$ be a boundary control system which is  assumed to be $\mathcal{Z}$-ISS where $\mathcal{Z}$ refers to either $L^{q}$ with $q<\infty$, $C$ or $\Reg$. Let $M\geq1$  and $\omega<0$ be such that for the associated semigroup $T$ it holds that $\|T(t)\|\leq M\ee^{\omega t}$ for all $t>0$. Furthermore, let $f:\Rp\times X\to X$ be continuous in the first and uniformly Lipschitz continuous in the second variable with Lipschitz constant $L_{f}>0$ and $f(t,0)=0$ for all $t\geq0$. If
$$\omega+ML_{f}<0,$$  
then the semilinear boundary control system $\Sigma(\mathfrak{A},\mathfrak{B},f)$ is $\mathcal{Z}$-ISS. More precisely, 
 for any $x_{0}\in X$ and $u\in \mathcal{Z}(\Rp;U)$ System \eqref{eq2:nonlinearBCS} has a unique global mild solution $x\in C([0,\infty);X)$. Furthermore, there exist $\beta\in \mathcal{KL}$ and a constant $\sigma>0$ such that for all $t>0$, $x_{0}\in X$ and $u\in \mathcal{Z}([0,t];U)$, 
\begin{equation}
\|x(t)\|_{X}\leq \beta(\|x_{0}\|,t)+\sigma\|u\|_{\infty,[0,t]},
\end{equation}
thus $\Sigma(A,B,f)$ is $\mathcal{Z}$-ISS.
\end{theorem}
\begin{proof}
The proof follows the lines of a standard technique for semilinear equations with (global) Lipschitz continuous nonlinearity. For fixed $u\in \mathcal{Z}_{\mathrm{loc}}(0,\infty;U)$ and $x_{0}\in X$, it follows from the assumed ISS that the mapping
\begin{equation*}
t\mapsto g(t):=T(t)x_{0}+\int_{0}^{t}T(t-s)Bu(s)\mathrm{d}s
\end{equation*}
is continuous from $[0,\infty)$ to $X$. Indeed, the continuity follows by \cite[Theorem 4.3.2]{Staf05} (noting that ISS implies $\mathcal{Z}$-admissibility/$\mathcal{Z}$-wellposedness). The existence of a mild solution to $\Sigma(A,B,f)$ is equivalent to the existence of a fix-point $x\in C([0,\infty);X)$ of 
\begin{equation*}
x(t)=g(t)+\int_{0}^{t}T(t-s)f(s,x(s))\mathrm{d}s,\qquad t\in[0,\infty).
\end{equation*}
The latter follows from \cite[Cor.~6.1.3]{Pazy83} by the assumptions on $f$ and the continuity of $g$. The ISS property can now be shown by a Gronwall-type argument: Since the linear boundary control system is ISS, there exists $\sigma>0$  such that 
$$\|g(t)\|\leq M\ee^{\omega t}\|x_{0}\|+\sigma\|u\|_{\mathcal{Z}([0,t];U)},\qquad t>0,$$
where $M$ and $\omega$ are chosen as in the statement of the theorem.
By the definition of the mild solution,
\begin{align*}
\|x(t)\|\leq{}& \|g(t)\|+\int_{0}^{t}M\ee^{(t-s)\omega}\|f(s,x(s)\|\mathrm{d}s\\
	\leq{}& \|g(t)\|+ML_{f}\ee^{t\omega}\int_{0}^{t}\ee^{-\omega s}\|x(s)\|\mathrm{d}s.
\end{align*}
Now Gronwall's  inequality implies that
\begin{align*}
\|x(t)\| \leq{}&  \|g(t)\|+\ee^{t\omega}\int_{0}^{t}\ee^{-\omega s}ML_{f}\|g(s)\|\ee^{ML_{f}(t-s)}\mathrm{d}s\\
\leq{}&\|g(t)\|+ML_{f}\ee^{t(\omega+ML_{f})}\left(\int_{0}^{t}M\ee^{-ML_{f}s}\mathrm{d}s\|x_{0}\|+\right.\\
{}&\left.+\sigma(\|u\|_{\mathcal{Z}([0,t];U)})\int_{0}^{t}\ee^{-\omega s-ML_{f}s}\mathrm{d}s\right)
\\
={}&M\ee^{t(\omega+ML_{f})}\|x_{0}\|+\left[ML_{f}\frac{\ee^{t\omega+ML_{f}t}-1}{\omega+ML_{f}}+1\right]\sigma\|u\|_{\mathcal{Z}([0,t];U)}.
\end{align*}
Since the coefficient of the second term on the right hand side is bounded in $t$, the assertion follows.
\end{proof}
It is trivially seen that the condition on the Lipschitz constant is in general sharp as the finite-dimensional example
\[\dot{x}=-x+2x+u,\]
with $f(x)=2x$, shows. On the other hand, the slight adaption $X=\R$, $A=1$, $f(x)=-2x$, $B=0$ shows, that the result is not optimal in the sense that the ``sign'' of the nonlinearity  is crucial for asymptotic stability.

\begin{theorem} \label{thm:main}
Let $(\mathfrak{A},\mathfrak{B},f)$ be a semilinear boundary control system with associated operators $A$ and $B_{0}$. Let the following be satisfied for the linear system $(\mathfrak{A},\mathfrak{B})$:
\begin{enumerate}[label=(\roman*)]
\item the operator $A=\mathfrak{A}|_{\ker{\mathfrak{B}}}$ is self-adjoint and bounded from above by $\omega_{A}\in \R$, i.e.\ $\langle Ax,x\rangle\leq \omega_{A}$ for all $x\in D(A)$,
\item $B\in\mathcal{L}(U, X_{-\frac{1}{2}})$, where $B:=(\mathfrak{A}-A)B_{0}$.
\end{enumerate} 
Furthermore,  the function $f:[0,\infty)\times X_{\frac{1}{2}}\to X$ satisfies the following properties
\begin{enumerate}[label=(\arabic*)]
	\item $f$ is locally H\"older continuous in the first and Lipschitz in the second variable, i.e.\ for any $(t,x)\in \Rp\times X_{\frac{1}{2}}$ there exists $L>0$, $\theta\in(0,1)$, $\rho>0$ such that 
	$$\|f(t,x)-f(s,y)\|\leq L(|t-s|^{\theta}+\|x-y\|_{\frac{1}{2}})$$
	for all $(s,t)$ in the ball $B_{\rho}(t,x)$ in $\Rp\times X$  with radius $\rho$ and centre $(t,x)$.
	\item there exists a continuous, nondecreasing function $k:\Rp\to\Rp$ such that $$\|f(t,x)\|\leq k(t)(1+\|x\|_{\frac{1}{2}}),\quad \forall (t,x)\in \Rp\times X.$$
	\item there exists  constants $m_{1},m_{2}\in \R$ such that  for any $(t, x)\in \Rp\times X_{\frac{1}{2}}$ it holds that $\langle f(t,x),x\rangle\in\R$  and \[\langle f(t,x),x\rangle \leq -m_{1}\langle Ax,x\rangle+m_{2}\|x\|^{2}.\]
	\item above constants satisfy the inequality
	\[ 1-m_{1}>0 \quad\text{and}\quad (1-m_{1})\omega_{A}+m_{2}<0. \]
\end{enumerate}
Then, for any $x_{0}\in X_{\frac{1}{2}}$ and $u\in W^{2,1}(\Rp;U)$ with $A_{-1}x_{0}+Bu(0)\in X$, the semilinear boundary control system \eqref{eq2:nonlinearBCS} has a unique mild solution $x$, which is classical on $(0,\infty)$, and $(\mathfrak{A},\mathfrak{B},f)$ is $L^{q}$-ISS for any $q\ge 2$. More precisely, for any $q\geq2$ 
there exist constants $C_{1},C_{2},\omega>0$ such that for all $(t,x_{0},u)\in \Rp\times X_{\alpha}\times W^{2,1}(\Rp;U)$  with $A_{-1}x_{0}+Bu(0)\in X$ the solution $x$ satisfies
$$\|x(t)\|_{X}\leq C_{1}\ee^{-\omega t} \|x_{0}\|_{X} +C_{2}\|u\|_{L^{q}(0,t;U)}.$$
\end{theorem}
\begin{proof} First note that --- upon considering $\tilde{\mathfrak{A}}=\mathfrak{A}-\omega_{A}-\epsilon$ and $\tilde{f}(s,x)=f(s,x)+(\omega_A+\epsilon)x$ we can without loss of generality assume that $\omega_{A}<0$ and thus that the semigroup is exponentially stable. \\
In order to show existence and uniqueness of the solutions, we closely follow the proof of the classical result in \cite[Theorem 6.3.1 and Theorem 6.3.3]{Pazy83} which has to be adapted to allow for boundary inputs $u$.  
Under the made assumptions on $A$ and $f$, it follows by \cite[Theorem 6.3.1]{Pazy83}, that the uncontrolled system, $u\equiv0$, has a unique local classical solution for any $x_{0}\in X_{\frac{1}{2}}$, which, by the assumption (2) and \cite[Theorem 6.3.3]{Pazy83}, extends to a global solution. The key argument for local existence \cite[Theorem 6.3.1]{Pazy83} is to consider the unique solution $y$ of 
\begin{equation}\label{eq1:lastthm}
y(t)=T(t)(-A)^{\frac{1}{2}}x_{0}+\int_{0}^{t}(-A)^{\frac{1}{2}}T(t-s)f(s,(-A)^{-\frac{1}{2}}y(s))\mathrm{d}s
\end{equation}
for $t\in[0,\tau]$, where $\tau>0$ and to show that $t\mapsto y(t)$ is H\"older continuous on $(0,\tau)$, so that the seeked solution is given by the solution of
\begin{align}\label{eq2:lastthm}
\dot{x}(t)={}&Ax(t)+f(t,(-A)^{-\frac{1}{2}}y(t)),
\\ x(0)={}&x_{0}.\notag
\end{align}
To apply an analogous reasoning in the controlled case, $u\neq0$, it remains to adapt \eqref{eq1:lastthm} and \eqref{eq2:lastthm} by adding the terms $\int_{0}^{t}(-A)^{\frac{1}{2}}T(t-s)Bu(s)\mathrm{d}s$ and $Bu(t)$ to the right-hand sides, respectively. Since $B\in\mathcal{L}(U;X_{-\frac{1}{2}})$, we have that $\tilde{B}:=(-A)^{-\frac{1}{2}}B\in\mathcal{L}(U;X)$ and thus
\begin{align*}
t\mapsto\int_{0}^{t}(-A)^{\frac{1}{2}}T(t-s)Bu(s)\mathrm{d}s={}&-\int_{0}^{t}AT(t-s)\tilde{B}u(s)\mathrm{d}s\\
={}&-\int_{0}^{t}T(t-s)\tilde{B}\dot{u}(s)\mathrm{d}s+T(t)\tilde{B}u(0)-\tilde{B}u(t)
\end{align*}
is a continuous function on $[0,\infty)$ and, by the analyticity of the semigroup, even H\"older continuous on $(0,\infty)$. Therefore, analogously to the proof of \cite[Theorem 6.3.1]{Pazy83}, we conclude that the equation
\begin{equation}\label{eq3:lastthm}
y(t)=T(t)(-A)^{\frac{1}{2}}x_{0}+\int_{0}^{t}(-A)^{\frac{1}{2}}T(t-s)\left[f(s,(-A)^{-\frac{1}{2}}y(s))+Bu(s)\right]\mathrm{d}s, 
\end{equation}
allows for a unique continuous solution $y:[0,\tau]\to X$ for some $\tau>0$  such that $t\mapsto f(t,(-A)^{-\frac{1}{2}}y(t))$ is H\"older continuous on $(0,\tau)$.
Therefore, and since $u\in W^{2,1}(\Rp;U)$ with $A_{-1}x_{0}+Bu(0)\in X$, the mild solution $x\in C([0,\tau];X)$ of 
\begin{align}\label{eq4:lastthm}
\dot{x}(t)={}&Ax(t)+f(t,(-A)^{-\frac{1}{2}}y(t))+Bu(t),
\\ x(0)={}&x_{0}.\notag
\end{align}
is in fact a classical solution on $(0,\tau)$, \cite[Cor.~4.3.3]{Pazy83} and \cite[Proposition 4.2.10]{TucsWeis09}. From the representation of the mild solution of \eqref{eq4:lastthm},
$$x(t)=T(t)x_{0}+\int_{0}^{t}T(t-s)\left[f(s,(-A)^{-\frac{1}{2}}y(s))+Bu(s)\right]\mathrm{d}s,$$
 it moreover follows that $x(t)=(-A)^{-\frac{1}{2}}y(t)$ and thus, $x$ is a mild solution of the original boundary control problem \eqref{eq2:nonlinearBCS} on $[0,\tau]$ and even a classical solution on $(0,\tau)$. From assumption (4), it follows that $x$ remains bounded in the $\|\cdot\|_{\frac{1}{2}}$-norm on $[0,\tau)$, so that, by iterating the argument, $x$ can be extended to a global solution, see \cite[Theorem 6.3.3]{Pazy83}.\smallskip
 
  We now show the $L^{q}$-ISS estimate. Let $x$ be the mild solution to an initial value $x_{0}\in  X_{\frac{1}{2}}$. Since $x$ is a classical solution on $(0,\infty)$, we have  for any $t>0$ that
\[ \frac{1}{2}\frac{{\rm d}}{{\rm d}t}\|x(t)\|^{2}=\langle Ax(t),x(t)\rangle +\langle f(t,x(t)),x(t)\rangle+\Re\langle u(t),B^{*}x(t)\rangle.\]
Therefore, by Assumption (3) and noting that 
\begin{equation}\label{eq:thm3}
\|x\|_{\frac{1}{2}}^{2}=\langle (-A)^{\frac{1}{2}}x,(-A)^{\frac{1}{2}}x\rangle=-\langle Ax,x\rangle\geq -\omega_{A}\|x\|^{2},
\end{equation}
it follows that for any $t>0$ and sufficiently small $\epsilon>0$
\begin{align}
\frac{1}{2}\frac{{\rm d}}{{\rm d}t}\|x(t)\|^{2}\leq{}& (1-m_{1})\langle Ax(t),x(t)\rangle+m_{2}\|x(t)\|^{2}+|\langle u(t),B^{*}x(t)\rangle_{U\times U}|,\notag\\
\leq{}& (1-m_{1}-\epsilon)\langle Ax(t),x(t)\rangle+m_{2}\|x(t)\|^{2}+\tfrac{1}{4\epsilon}\|B^{*}\|_{\mathcal{L}(X_{\frac{1}{2}},U)}^{2}\|u(t)\|^{2},\notag\\
\leq{}&((1-m_{1}-\epsilon)\omega_{A}+m_{2})\|x(t)\|^{2}+\tfrac{1}{4\epsilon}\|B^{*}\|_{\mathcal{L}(X_{\frac{1}{2}},U)}^{2}\|u(t)\|^{2},\label{eq:thmmain}
\end{align}
where we used \eqref{eq:thm3} and Assumption (4) in the last inequality. 
Gronwall's inequality now yields the assertion for $q=2$ and an additional application of H\"older's inequality the one for $q>2$.
\end{proof}
\begin{rem}
Theorem \ref{thm:main} is a generalization of the result in \cite{ZhenZhu17} where only the Laplacian with Robin/Neumann boundary control (excluding Dirichlet control) in one spatial variable was considered and the assumptions on $f$ were tuned to guarantee the existence of classical solutions. We decided to give a full proof (or at least a sketch of the necessary adaptions from \cite{Pazy83}) of the existence of solutions for the convenience of the reader, but also since the classical literature on semilinear pde's does not cover the presence of the inputs\footnote{at least the author is not aware of any explicit reference in this operator theoretic framework.}. The assumption that the inputs should lie $W^{2,1}(\R)$ with the additional property that $A_{-1}x_{0}+Bu(0)\in X$ is clearly tuned in order to guarantee for classical solutions (in $(0,\infty)$), cf.\cite[Prop.4.2.10]{TucsWeis09}. This, however, can be weakened with a more careful analysis on the regularity of the solutions and by deriving \eqref{eq:thmmain} only for almost every $t>0$. 
Although our proof follows standard arguments in the semigroup approach to semilinear equations instead, the derivation of the ISS estimate can be seen as abstraction of the procedure in \cite{ZhenZhu17}. 
Recall that it is well-known that the corresponding boundary operator $B$ in the situation of Neumann or Robin control in \cite{ZhenZhu17} satisfies the condition $B\in\mathcal{L}(U, X_{-\frac{1}{2}})$, see also Example \ref{ex3:heatNeumann}.
\end{rem}

\begin{example}[Semilinear parabolic equation with cubic nonlinearity]
\label{ex:semilinear}
 Let $\Omega\subset\R^{n}$ with $n\in\{1,2,3\}$. Under the setting of Example \ref{ex3:heatNeumann} consider
\begin{align*}
	\dot{x}(\xi,t)={}&\Delta x(\xi,t)-ax(\xi,t)-x(\xi,t)^{3}+d(\xi,t),&(\xi,t)\in\phantom{\partial}\Omega\times (0,\infty)\\
	\frac{\partial x}{\partial \nu}(\xi,t)={}&u(\xi,t)&(\xi,t)\in\partial \Omega\times (0,\infty)\\
	x(\xi,0)={}&x_{0}(\xi), &\xi\in\phantom{\partial}\Omega\phantom{\times (0,\infty)}
\end{align*}
which establishes a semilinear BCS $(\mathfrak{A},\mathfrak{B},f)$ with $f(x)=-x^{3}$ and the same operators $\mathfrak{A},\mathfrak{B}$, $A$, $B$ as in Example \ref{ex3:heatNeumann}. As seen in the previous example, $(\mathfrak{A},\mathfrak{B})$ is a linear boundary control system for $d=0$ and, in the generalized sense of Remark \ref{rem1}, for $d\neq0$. The conditions (i) and (ii) of Theorem \ref{thm:main} are satisfied with $\omega_{A}=-a$. Conditions (1) and (2) both follow from the Sobolev embedding $W^{1,2}(\Omega)\subseteq L^{6}(\Omega)$ valid for $n\in\{1,2,3\}$, see e.g.\ \cite{Adams}, and the fact that $X_{\frac{1}{2}}=W^{1,2}(\Omega)$, see e.g.\ \cite{LasiTrigI00}.

\end{example}

\section{Concluding remarks \& Outlook}
In the situation of Dirichlet boundary control and the choice $X=L^{2}(\Omega)$ for the state space, it is well-known that an $L^{2}$-ISS-estimate (in time) cannot be expected. More precisely, even for a linear heat equation the input operator represented by Dirichlet boundary control is not $L^{2}$-admissible if the state space is $L^{2}(\Omega)$, see \cite[p.~217]{Lion68} for a counterexample. Instead, as we have seen in Example \ref{ex:dirichlet}, we only have $L^{p}$-ISS for $p>4$ in general, see also \cite[Proposition 5.1]{FabrPuelZuaz95} for another proof in the case that $p=\infty$. Therefore, the results of Section \ref{sec4} cannot be applied and the situation becomes more evolved. The question is if Lyapunov arguments such as used in Theorem \ref{thm:main} can at all be used to assess ISS in situations which are not $L^{2}$-ISS. A work-around --- typical in the theory of linear $L^2$-wellposed systems \cite{TucsWeis09}--- is as follows: If in the setting of Example \ref{ex:semilinear} one considers Dirichlet boundary control instead of Neumann boundary control, we could change the considered state space $X$ to be the Sobolev space $H^{-1}(\Omega)$ in order to obtain $L^{2}$-ISS, i.e.\ 
\begin{equation*}
\|x(t)\|_{H^{-1}(\Omega)}\lesssim \ee^{-a t}\|x_{0}\|_{H^{-1}(\Omega)}+\|u\|_{L^{2}(0,t;L^{2}(\partial\Omega))}.
\end{equation*} 
On the other hand, if we aim for $L^{\infty}$-ISS estimates only, other techniques may be more suitable; such as the maximum principle methods in \cite{KaraKrstMiro19}.
These methods, however, seem to be practical only for $L^{\infty}$-ISS estimates.

\section*{Acknowledgements}
The author would like to thank Marc Puche (Hamburg), Andrii Mironchenko (Passau) and Hans Zwart (Twente) for fruitful discussions on semilinear equations. He is also grateful to the anonymous referee for the careful reading of the manuscript. Finally he is indebted to Hafida Laasri, Joachim Kerner and Delio Mugnolo (Hagen) for organizing the inspiring workshop {\it Control theory of infinite-dimensional systems} in Hagen, Germany, in January 2018, from which the idea for this article originated.


\def\cprime{$'$}

\end{document}